\documentclass[12pt]{amsart}
\usepackage{amssymb,latexsym,amsmath,amscd,amsthm,graphicx, color}
\usepackage[all]{xy}
\usepackage{pgf,tikz}
\usepackage{mathrsfs}
\usetikzlibrary{arrows}
\definecolor{uuuuuu}{rgb}{0.26666666666666666,0.26666666666666666,0.26666666666666666}
\definecolor{xdxdff}{rgb}{0.49019607843137253,0.49019607843137253,1.}
\definecolor{ffqqqq}{rgb}{1.,0.,0.}
\definecolor{ffqqqq}{rgb}{1.,0.,0.}
\definecolor{ffxfqq}{rgb}{1.,0.4980392156862745,0.}

\definecolor{uuuuuu}{rgb}{0.26666666666666666,0.26666666666666666,0.26666666666666666}
\definecolor{qqwuqq}{rgb}{0.,0.39215686274509803,0.}
\definecolor{zzttqq}{rgb}{0.6,0.2,0.}
\definecolor{xdxdff}{rgb}{0.49019607843137253,0.49019607843137253,1.}
\definecolor{qqqqff}{rgb}{0.,0.,1.}
\definecolor{cqcqcq}{rgb}{0.7529411764705882,0.7529411764705882,0.7529411764705882}
\definecolor{sqsqsq}{rgb}{0.12549019607843137,0.12549019607843137,0.12549019607843137}

\theoremstyle{plain}

\newtheorem{theorem}[subsection]{Theorem}

\newtheorem{lemma}[subsection]{Lemma}
\newtheorem{defi}[subsection]{Definition}
\newtheorem{prop}[subsection]{Proposition}

\theoremstyle{definition}

\newtheorem{exam}[subsection]{Example}

\newtheorem{remark}[subsection]{Remark}

\newtheorem{claim}[subsection]{Claim}

\newcommand{\uu}{\cup}
\newcommand{\ii}{\cap}

\newcommand{\sci}{\subset}

\newcommand{\nin}{\notin}
\newcommand{\es}{\emptyset}
\newcommand{\set}[1]{\{#1\}}

\newcommand{\ga}{\alpha}
\newcommand{\gb}{\beta}

\newcommand{\gq}{\theta}

\newcommand{\tbf}{\textbf}
\newcommand{\tit}{\textit}

\newcommand{\D}[1]{\mathbb{#1}}

\newcommand{\te}{\text}

\newcommand{\ol}{\overline}

\newcommand{\pa}{\partial}

\begin{document}
To appear, Uniform Distribution Theory
\title{Quantization for a mixture of uniform distributions associated with probability vectors}

 \author{Mrinal Kanti Roychowdhury}
\address{School of Mathematical and Statistical Sciences\\
University of Texas Rio Grande Valley\\
1201 West University Drive\\
Edinburg, TX 78539-2999, USA.}
\email{mrinal.roychowdhury@utrgv.edu}

 \author{Wasiela Salinas }
\address{School of Mathematical and Statistical Sciences\\
University of Texas Rio Grande Valley\\
1201 West University Drive\\
Edinburg, TX 78539-2999, USA.}
\email{wsalinas47@gmail.com}

\subjclass[2010]{60Exx, 94A34.}
\keywords{Mixed distribution, uniform distribution, optimal sets, quantization error, quantization dimension, quantization coefficient}

\date{}
\maketitle

\pagestyle{myheadings}\markboth{Mrinal Kanti Roychowdhury and Wasiela Salinas}{Quantization for a mixture of uniform distributions associated with probability vectors}

\begin{abstract} The basic goal of quantization for probability distribution is to reduce the number of values, which is typically uncountable, describing a probability distribution to some finite set and thus approximation of a continuous probability distribution by a discrete distribution.
Mixtures of probability distributions, also known as mixed distributions, are an exciting new area for optimal quantization. In this paper, we investigate the optimal quantization for three different mixed distributions generated by uniform distributions associated with probability vectors.
\end{abstract}

\section{Introduction}

Continuous-valued signals can take any real value either in the entire range of real numbers or in a range limited by some system constraints. In either of the two cases, an uncountably infinite set of values is required to represent the signal values.
If a signal has to be processed or stored digitally, each of its values must be representable by a finite number of bits. Thus, all values together have to form a finite countable set. A signal consisting only of such discrete values is said to be quantized. The process of transformation of a continuous-valued signal into a discrete-valued one is called `quantization'.  It has broad application in engineering and technology (see \cite{GG, GN, Z}).
For mathematical treatment of quantization one is referred to Graf-Luschgy's book (see \cite{GL1}).
 Let $\D R^d$ denote the $d$-dimensional Euclidean space equipped with the Euclidean norm $\|\cdot\|$, and let $P$ be a Borel probability measure on $\D R^d$. Then, the $n$th \textit{quantization
error} for $P$, with respect to the squared Euclidean distance, is defined by
\begin{equation*} \label{eq1} V_n:=V_n(P)=\inf \Big\{V(P; \ga) : \ga \subset \mathbb R^d, \text{ card}(\ga) \leq n \Big\},\end{equation*}
where $V(P; \ga):= \int \min\limits_{a\in\alpha} \|x-a\|^2 dP(x)$ represents the distortion error for $P$ due to the set $\ga$.
A set $\ga\sci \D R^d$ is called an optimal set of $n$-means for $P$ if $V_n(P)=V(P; \ga)$. It is known that for a continuous Borel probability measure an optimal set of $n$-means always has exactly $n$-elements (see \cite{GL1}). Optimal sets of $n$-means for different probability distributions were calculated by several authors, for example, one can see \cite{CR, DR1, DR2, GL2, L1, R1, R2, R3, R4, R5, RR1}. The number
\[\lim_{n\to \infty}  \frac{2\log n}{-\log V_n(P)},\] if it exists, is called the \tit{quantization dimension} of the probability measure $P$, and is denoted by $D(P)$; on the other hand, for any $s\in (0, +\infty)$, the number $\lim\limits_{n\to\infty} n^{\frac 2 s} V_n(P)$, if it exists, is called the $s$-dimensional \tit{quantization coefficient} for $P$ (see \cite{GL1, P}).

Let us now state the following proposition (see \cite{GG, GL1}):
\begin{prop} \label{prop0}
Let $\ga$ be an optimal set of $n$-means for $P$, and $a\in \ga$. Then,

$(i)$ $P(M(a|\ga))>0$, $(ii)$ $ P(\partial M(a|\ga))=0$, $(iii)$ $a=E(X : X \in M(a|\ga))$,
where $M(a|\ga)$ is the Voronoi region of $a\in \ga , $ i.e.,  $M(a|\ga)$ is the set of all elements $x$ in $\D R^d$ which are closest to $a$ among all the elements in $\ga$.
\end{prop}

Proposition~\ref{prop0} says that if $\ga$ is an optimal set and $a\in\ga$, then $a$ is the \tit{conditional expectation} of the random variable $X$ given that $X$ takes values in the Voronoi region of $a$.
The following theorem is known.
\begin{theorem} [see \cite{RR2}] \label{th31}   Let $P$ be a uniform distribution on the closed interval $[a, b]$. Then, the optimal set $n$-means is given by $ \ga_n:=\set{a+\frac {2i-1}{2n}(b-a) : 1\leq i\leq n}$, and the corresponding quantization error is
$V_n:=V_n(P)=\frac{(a-b)^2}{12 n^2}.$
\end{theorem}

\begin{theorem} \label{th32}
Let $\ga_n$ be an optimal set of $n$-means for a uniform distribution on the unit circular arc $S$ given by
\[S:=\set{(\cos \gq, \sin \gq) : \ga\leq \gq\leq \gb},\]
where $0\leq \ga<\gb\leq 2\pi$. Then,
\[\ga_n:=\Big\{ \frac {2n}{\gb-\ga} \sin(\frac{\gb-\ga}{2n}) \Big(\cos \Big(\ga+(2j-1){\frac{\gb-\ga}{2n}}\Big), \   \sin  \Big(\ga+(2j-1){\frac{\gb-\ga}{2n}}\Big)\Big) : j=1, 2, \cdots, n \Big  \} \]
forms an optimal set of $n$-means, and the corresponding quantization error is given by
\[V_n=\frac{(\ga-\gb)^2-2n^2+2n^2\cos\frac{\ga-\gb}{n}}{(\ga-\gb)^2}.\]
\end{theorem}
\begin{proof}
Notice that $S$ is an arc of the unit circle $x_1^2+x_2^2=1$ which subtends a central angle of $\gb-\ga$ radian, and the probability distribution is uniform on $S$. Hence, the density function is given by $f(x_1, x_2)=\frac 1{\gb-\ga}$ if $(x_1, x_2)\in S$, and zero, otherwise. Thus, the proof follows in the similar way as the proof in the similar theorem in \cite{RR2}.
\end{proof}

Mixed distributions are an exciting new area for optimal quantization. For any two Borel probability measures $P_1$ and $P_2$, and $p\in (0, 1)$, if $P:=p P_1+(1-p)P_2$, then the probability measure $P$ is called the \tit{mixture} or the \tit{mixed distribution} generated by the probability measures $(P_1, P_2)$ associated with the probability vector $(p, 1-p)$. Such kind of problems has rigorous applications in many areas including signal processing. For example, while driving long distances, we have seen sometimes cellular signals get cut off. This happens because of being far away from the tower, or there is no tower nearby to catch the signal. In optimal quantization for mixed distributions one of our goals is to find the exact locations of the towers by giving different weights, also called importance, to different portions of a path.

The following theorem about the quantization dimension for the mixed distributions is well-known. For some more details please see \cite[Theorem~2.1]{L}.
\begin{theorem}  \label{the01}
Let $P_1$ and $P_2$ be any two Borel probability measures on $\D R^d$ such that both $D(P_1)$ and $D(P_2)$ exist. If $P=pP_1+(1-p)P_2$, where $0<p<1$, then $D(P)=\max\set{D(P_1), D(P_2)}$.
\end{theorem}

In this paper, in Section~\ref{sec2}, we have considered a mixed distribution generated by two uniform distributions on a circle and on one of its diameters associated with the probability vector $(\frac 12, \frac 12)$. For this mixed distribution, in Theorem~\ref{Th0}, we have explicitly determined the optimal sets of $n$-means and the $n$th quantization errors for all positive integers $n\geq 2$. In Proposition~\ref{prop90}, we have proved that the quantization dimension $D(P)$ of the mixed distribution is one, which supports Theorem~\ref{the01} because $D(P_1)=D(P_2)=1$, and the quantization coefficient exists as a finite positive number which equals $\frac{3}{8} \left(4+\pi ^2\right)$. Optimal sets of $n$-means and the $n$th quantization errors are calculated, in Section~\ref{sec3}, for the mixture of two uniform distributions on two disconnected line segments $[0, \frac 12]$ and $[\frac 34, 1]$ associated with the probability vector $(\frac 34, \frac 14)$, and in Section~\ref{sec4}, for the mixture of two uniform distributions on two connected line segments $[0, \frac 12]$ and $[\frac 12, 1]$ associated with the probability vector $(\frac 34, \frac 14)$. We would like to mention that in these two sections, to determine the optimal sets of $n$-means and the $n$th quantization errors for the mixed distributions we need to take the help of two different sequences $\set{a(n)}_{n=1}^\infty$ given by Definition~\ref{defi59}, and Definition~\ref{defi60}. If the probability vector $(\frac 34, \frac 14)$ is replaced by some other probability vector $(p, 1-p)$, where $0<p<1$, what will be the two such sequences are not known yet. In fact, optimal sets of $n$-means and the $n$th quantization errors are not known yet for a more general mixed distribution.

\section{Quantization for a mixed distribution on the circles including a diameter} \label{sec2}

Let $i$ and $j$ be the unit vectors in the positive directions of the $x_1$- and $x_2$-axes, respectively. By the position vector $a$ of a point $A$, it is meant that $\overrightarrow{OA}=a$. We will identify the position vector of a point $(a_1, a_2)$ by $(a_1, a_2):=a_1 i +a_2 j$, and apologize for any abuse in notation. For any two position vectors $  a:=( a_1, a_2)$ and $  b:=( b_1, b_2)$, we write $\rho(a, b):=\|(a_1, b_1)-(a_2, b_2)\|^2=(a_1-a_2)^2 +(b_1-b_2)^2$, which gives the squared Euclidean distance between the two points $(a_1, a_2)$ and $(b_1, b_2)$.  Let $P$ and $Q$ belong to an optimal set of $n$-means for some positive integer $n$, and let $D$ be a point on the boundary of the Voronoi regions of the points $P$ and $Q$. Since the boundary of the Voronoi regions of any two points is the perpendicular bisector of the line segment joining the points, we have
$|\overrightarrow{DP}|=|\overrightarrow{DQ}|, \te{ i.e., } (\overrightarrow{DP})^2=(\overrightarrow{DQ})^2$ implying
$(  p-  d)^2=(  q-  d)^2$, i.e., $\rho(  d,   p)-\rho(  d,   q)=0$. We call such an equation a \tit{canonical equation}.
By $E(X)$ and $V:=V(X)$, we represent the expectation and the variance of a random variable $X$ with respect to the probability distribution under consideration.

Let $P_1$ be the uniform distribution defined on the circle $x_1^2+x_2^2=1$ with center $O(0, 0)$, and $P_2$ be the uniform distribution on one of its diameters. Let us denote the diameter by $L_1$ and the circle by $L_2$. Without any loss of generality, we can assume that the diameter is horizontal, i.e., the diameter is represented by $L_1:=\set{(x_1, 0) : -1\leq x_1\leq 1}$ which intersects the circle at the two points $A(-1, 0)$ and $B(0, 1)$. Let $L$ be the path formed by the circle and the diameter $AB$. Thus, we have $L=L_1\uu L_2$, where
\[ L_1=\set{(t, 0) : -1\leq t\leq 1}, \te{ and } L_2=\set{(\cos \gq, \sin \gq) : 0\leq \gq\leq 2\pi}.\]
Let $s$ represent the distance of any point on $L$ from the origin tracing along the boundary $L$ in the positive direction of the $x_1$-axis, and in the counterclockwise direction.  Thus, $s=1$ represents the point $B(1, 0)$, $s=1+\frac \pi 2$ represents the point $(0,-1)$, and so on. Take the mixed distribution $P$ as
\[P:=\frac 12 P_1+\frac 12 P_2, \]
i.e., $P$ is generated by $(P_1, P_2)$ associated with the probability vector $(\frac 12, \frac 12)$.  For this mixed distribution $P$ in this section, we determine the optimal sets of $n$-means and the $n$th quantization errors for all $n\in \D N$.
The probability density function (pdf) $f(x_1, x_2)$ for the mixed distribution $P$ is given by
\begin{align*}
 f(x_1, x_2)=\left\{\begin{array}{cc}
\frac 1 4 & \te{ if } (x_1, x_2) \in L_1,\\
\frac  1 {4 \pi} & \te{ if }  (x_1, x_2) \in L_2.
\end{array}\right.
\end{align*}
On $L_1$ we have $ds=\sqrt{(\frac {dx_1}{dt})^2 +(\frac {dx_2}{dt})^2} \, dt=dt$ yielding $dP(s)=P(ds)=f(x_1, x_2)ds=\frac 14 dt$. Similarly, on $L_2$, we have $ds=d\gq$  yielding $dP(s)=P(ds)=f(x_1, x_2)ds=\frac 1{4 \pi} d\gq$.
\begin{lemma}
Let $X$ be a continuous random variable with mixed  distribution taking values on $L$. Then,
\[E(X)=(0, 0) \te{ and } V:=V(X)=\frac 23 .\]
\end{lemma}
\begin{proof} We have,
\begin{align*} \label{eq1}
&E(X)  =\int_L(x_1 i+x_2 j) dP=\frac 1 4\int_{L_1}  (t, 0) \,dt +\frac 1{4 \pi} \int_{L_2}(\cos \gq, \sin \gq) \,d\gq=(0,0).
\end{align*}  To calculate the variance, we know that $V(X)=E\|X-E(X)\|^2$, which implies
\begin{align*} V(X)=\frac 1 4\int_{L_1}  \rho((t, 0), (0, 0))\,dt +\frac 1{4 \pi} \int_{L_2}\rho((\cos \gq, \sin \gq), (0, 0)) \,d\gq=\frac{2}{3}.
\end{align*}
Thus, the lemma is yielded.
\end{proof}

\begin{remark} \label{remark45} Using the standard theory of probability, for any $(a, b) \in \D R^2$, we have
\begin{align*}
&E\|X-(a, b)\|^2=\int_L\|(x_1, x_2)-(a, b)\|^2dP=V(X)+  \|(a, b)-(0, 0)\|^2,
\end{align*}
which is minimum if $(a, b)=(0, 0)$, and the minimum value is $V(X)$.
Thus, we see that the optimal set of one-mean is the set $\set{(0, 0)}$, and the corresponding quantization error is the variance $V:=V(X)$ of the random variable $X$ (see Figure~\ref{Fig0}~$(i)$).
\end{remark}

\begin{figure}
\begin{tikzpicture}[line cap=round,line join=round,>=triangle 45,x=1.0cm,y=1.0cm]
\clip(-2.1431409553721457,-2.454073612323069) rectangle (2.119731706345336,2.1580023224116287);
\draw(0.,0.) circle (2.cm);
\draw (-1.9739043510998817,0.)-- (2.0260956489001183,0.);
\draw [dash pattern=on 1pt off 1pt,color=ffqqqq] (0.,-2.)-- (0.,2.);
\begin{scriptsize}
\draw [fill=ffqqqq] (0.,-2.) circle (0.5pt);
\draw [fill=ffqqqq] (0.,0.) circle (1.5pt);
\draw[] (0.00096496167399,-2.330438200585347174) node { $(i)$ };
\end{scriptsize}
\end{tikzpicture}
\begin{tikzpicture}[line cap=round,line join=round,>=triangle 45,x=1.0cm,y=1.0cm]
\clip(-2.1431409553721457,-2.454073612323069) rectangle (2.119731706345336,2.1580023224116287);
\draw(0.,0.) circle (2.cm);
\draw (-1.9739043510998817,0.)-- (2.0260956489001183,0.);
\draw [dash pattern=on 1pt off 1pt,color=ffqqqq] (0.,-2.)-- (0.,2.);
\draw (-1.1435572534639595,0.5206572340693282) node[anchor=north west] {$p_1$};
\draw (1.0701432039302117,0.5037587572953268) node[anchor=north west] {$p_2$};
\begin{scriptsize}
\draw [fill=ffqqqq] (0.,-2.) circle (0.5pt);
\draw [fill=ffqqqq] (-1.13662,0.) circle (1.5pt);
\draw [fill=ffqqqq] (1.13662,0.) circle (1.5pt);
\draw[] (0.00096496167399,-2.330438200585347174) node { $(ii)$ };
\end{scriptsize}
\end{tikzpicture}
\begin{tikzpicture}[line cap=round,line join=round,>=triangle 45,x=1.0cm,y=1.0cm]
\clip(-2.1431409553721457,-2.454073612323069) rectangle (2.119731706345336,2.1580023224116287);
\draw(0.,0.) circle (2.cm);
\draw (-1.9739043510998817,0.)-- (2.0260956489001183,0.);
\draw [dash pattern=on 1pt off 1pt,color=ffqqqq] (0.,-2.)-- (0.,2.);
\draw (0.00000764394213067,2.009114050825449) node[anchor=north west] {$p_1$};
\draw (-1.217354207011962,0.06578922181530106) node[anchor=north west] {$p_2$};
\draw (1.1377371110262169,0.09958617536330364) node[anchor=north west] {$p_3$};
\draw [dash pattern=on 1pt off 1pt,color=ffqqqq] (-1.53392,1.28339)-- (0.,0.);
\draw [dash pattern=on 1pt off 1pt,color=ffqqqq] (1.5274782908008957,1.2910499878517385)-- (0.,0.);
\begin{scriptsize}
\draw [fill=ffqqqq] (0.,-2.) circle (0.5pt);
\draw [fill=ffqqqq] (0.,1.75488) circle (1.5pt);
\draw [fill=ffqqqq] (-1.18781,-0.283581) circle (1.5pt);
\draw [fill=ffqqqq] (1.18781,-0.283581) circle (1.5pt);
\draw [fill=ffqqqq] (-1.53392,1.28339) circle (0.5pt);
\draw [fill=ffqqqq] (1.53392,1.28339) circle (0.5pt);
\draw[] (0.00096496167399,-2.330438200585347174) node { $(iii)$ };
\end{scriptsize}
\end{tikzpicture}
\caption{ } \label{Fig0}
\end{figure}

\begin{prop} \label{prop2}
The set $\set{(-\frac 1 4 -\frac 1 {\pi}, 0), (\frac 1 4 +\frac 1 {\pi}, 0)}$ forms the optimal set of two-means, and the corresponding quantization error is given by $V_2=0.343691$.
\end{prop}

\begin{proof}
Since $P$ is a mixed distribution giving the equal weights to both the component probabilities $P_1$ and $P_2$, and the path $L$ is symmetric with respect to the $x_2$-axis, without going into much calculation, we can assume that the boundary of the Voronoi regions of the two points in an optimal set of two-means lies along the $x_2$-axis. Thus, the optimal set of two-means is given by $\set{p_1, p_2}$ (see Figure~\ref{Fig0}~$(ii)$), where
\begin{align*}
p_1&=E(X : X \in \ol{AO } \ \uu (\te{left half of the circle}))=\frac{\frac{1}{4} \int_{-1}^0 (x,0) \, dx+\frac 1 {4\pi} \int_{\frac{\pi }{2}}^{\frac{3 \pi }{2}} (\cos\gq ,\sin\gq ) \, d\theta}{\frac{1}{4} \int_{-1}^0  \, dx+\frac 1{4\pi} \int_{\frac{\pi }{2}}^{\frac{3 \pi }{2}}  \, d\theta}\\
&=(-\frac 1 4 -\frac 1 {\pi}, 0),
\end{align*}
and similarly, $p_2=(\frac 1 4+\frac 1 {\pi}, 0)$. The quantization error for two-means is given by
\[V_2=2\Big(\frac{1}{4} \int_{-1}^0 \rho((x,0), p_1) \, dx+\frac 1 {4\pi} \int_{\frac{\pi }{2}}^{\frac{3 \pi }{2}} \rho((\cos\gq ,\sin\gq ), p_1)\, d\theta\Big)=0.343691.\]
Thus, the proposition is yielded.
\end{proof}

The following proposition gives the optimal set of three-means (see Figure~\ref{Fig0}~$(iii)$). The proof follows in the similar way as Proposition~\ref{prop4} which is given later.

\begin{prop}\label{prop3}
The set $\set{(0, 0.877439), (-0.593906, -0.14179), (0.593906, -0.14179)}$ forms an optimal set of three-means, and the corresponding quantization error is given by $V_3=0.2386$.
\end{prop}

\begin{prop} \label{prop4}
The set $\set{(0, 0.90407), (-0.633881, 0), (0, -0.90407), (0.633881, 0)}$ forms an optimal set of four-means, and the corresponding quantization error is given by $V_4=0.163013$.
\end{prop}

\begin{proof} Let $\ga:=\set{p_1, p_2, p_3, p_4}$ be an optimal set of four-means. The following cases can arise:

\tit{Case~1. $\ga$ contains one point from $L_1$, the Voronoi region of which does not contain any point from $L_2$.}

In this case, we can assume that $p_1, p_2, p_3, p_4$ can be located as shown in Figure~\ref{Fig01}~$(i)$. Let the boundary of the Voronoi regions of $p_1$ and $p_2$ intersect $L_2$ at the point $d_1$ given by the parametric value $\gq=\pi-b$, where $0<b<\frac {\pi}2$, and the boundary of the Voronoi regions of $p_2$ and $p_3$ intersect $L_1$ at the point $d_2$ given by $x_1=-a$, where $0<a<1$. Thus, due to symmetry, we have
\begin{align*} p_1&=\frac{\int_b^{\pi -b} (\cos \theta ,\sin \theta )\, d\theta }{\int_b^{\pi -b}   d\theta }=\Big(0,\frac{2 \cos b}{\pi -2 b}\Big),\\
p_2&=\frac{\frac{1}{4} \int_{-1}^{-a} (x,0) \, dx+\frac1{4\pi} \int_{\pi -b}^{\frac {3\pi}2} (\cos \theta,\sin \theta) \, d\theta}{\frac{1}{4} \int_{-1}^{-a} dx+\frac1{4\pi} \int_{\pi -b}^{\frac {3\pi}2}  d\theta}=\Big(\frac{-\pi  a^2+2 \sin b+\pi +2}{\pi  (2 a-3)-2 b},-\frac{2 \cos b}{-2 \pi  a+2 b+3 \pi }\Big), \\
p_3&=(0, 0), \quad d_1=(-\cos b, \sin b), \te{ and } d_2=(-a, 0).
\end{align*}
Thus, solving the canonical equations $\rho(d_1, p_1)-\rho(d_1, p_2)=0$, and $\rho(d_2, p_2)-\rho(d_2, p_3)=0$, we have $a=0.377997$, $b=0.678642$. Hence, putting the values of $a$ and $b$ we have, $p_1=(0, 0.872524)$, $p_2=(-0.707525, -0.185184)$, and $p_3=(0, 0)$, and so, due to symmetry, $p_4=(0.707525, -0.185184)$. The corresponding distortion error is given by
\begin{align*}
V(P; \ga)&=\frac 1{4\pi} \int_b^{\pi -b} \rho((\cos \theta ,\sin \theta ), p_1)\, d\theta+2\Big(\frac{1}{4} \int_{-1}^{-a} \rho((x,0), p_2) \, dx\\
&\qquad  +\frac1{4\pi} \int_{\pi -b}^{\frac {3\pi}2} \rho((\cos \theta,\sin \theta), p_2) \, d\theta\Big)
+\frac{1}{4} \int_{-a}^{a} \rho((x,0), p_3) \, dx=0.21596.
\end{align*}

\tit{Case~2. $\ga$ does not contain any point from $L_1$, the Voronoi region of which does not contain any point from $L_2$.}

In this case, we can assume that $p_1, p_2, p_3, p_4$ can be located as shown in Figure~\ref{Fig01}~$(ii)$. Let the boundary of the Voronoi regions of $p_1$ and $p_2$ intersect $L_2$ at the point $d_1$ given by the parametric value $\gq=\pi-b$, where $0<b<\frac {\pi}2$. Thus, due to symmetry, we have
\begin{align*} p_1&=\frac{\int_b^{\pi -b} (\cos \theta ,\sin \theta )\, d\theta }{\int_b^{\pi -b}   d\theta }=\Big(0,\frac{2 \cos b}{\pi -2 b}\Big),\\
p_2&=\frac{\frac{1}{4} \int_{-1}^0 (x,0) \, dx+\frac1{4\pi} \int_{\pi -b}^{\pi+b} (\cos \theta,\sin \theta) \, d\theta}{\frac{1}{4} \int_{-1}^0 dx+\frac1{4\pi} \int_{\pi -b}^{\pi+b}  d\theta}=\Big(-\frac{4 \sin b+\pi }{4 b+2 \pi },0\Big), \te{ and}\\
d_1&=(-\cos b, \sin b).
\end{align*}
Thus, solving the canonical equations $\rho(d_1, p_1)-\rho(d_1, p_2)=0$, we have $b=0.800791$. Hence, putting the values of $b$, we have, $p_1=(0, 0.90407)$, $p_2=(-0.633881, 0)$, and so, due to symmetry, $p_3=(0.633881, 0)$, and $p_4=(0, -0.90407)$. The corresponding distortion error is given by
\begin{align*}
V(P; \ga)&=2\Big(\frac 1{4\pi}\int_{b}^{\pi -b} \rho((\cos \theta ,\sin \theta ), p_1)\, d\theta +\frac{1}{4} \int_{-1}^{0} \rho((x,0), p_2) \, dx\\
&\qquad  +\frac1{4\pi} \int_{\pi -b}^{\pi+b} \rho((\cos \theta,\sin \theta), p_2) \, d\theta \Big)=0.163013.
\end{align*}
Comparing Case~1 and Case~2, we see that if $\ga$ contains only one point from $L_1$, the Voronoi regions of which does not contain any point from $L_2$, then the distortion error is larger than the distortion error obtained in Case~2. Similarly, we can show that if $\ga$ contains more than one point from $L_1$, the Voronoi regions of which do not contain any point from $L_2$, then the distortion error is larger than the distortion error obtained in Case~2. Considering all the above cases, we see that the distortion error in Case~2 is the smallest. Hence, the points in $\ga$ obtained in Case~2 form an optimal set of four-means, and the corresponding quantization error is given by $V_4=0.163013$.
Thus, the proof of the proposition is complete.
 \end{proof}

\begin{figure}
\begin{tikzpicture}[line cap=round,line join=round,>=triangle 45,x=1.0cm,y=1.0cm]
\clip(-2.1431409553721457,-2.454073612323069) rectangle (2.119731706345336,2.1580023224116287);
\draw(0.,0.) circle (2.cm);
\draw (-1.9739043510998817,0.)-- (2.0260956489001183,0.);
\draw [dash pattern=on 1pt off 1pt,color=ffqqqq] (0.,-2.)-- (0.,2.);
\draw (0.00043764394213067,2.109114050825449) node[anchor=north west] {$p_1$};
\draw (-1.4539328818479804,0.03060140762329074) node[anchor=north west] {$p_2$};
\draw (-0.03125778637987325,0.46996180374732427) node[anchor=north west] {$p_3$};
\draw (1.3743157858622352,0.0812968379452946) node[anchor=north west] {$p_4$};
\draw (-2.108871850680196,1.6697536547014158) node[anchor=north west] {$d_1$};
\draw (-0.869384671531936,0.4868602805213256) node[anchor=north west] {$d_2$};
\draw [dash pattern=on 1pt off 1pt,color=ffqqqq] (1.55685,1.25547)-- (0.755994,0.);
\draw [dash pattern=on 1pt off 1pt,color=ffqqqq] (-0.755994,0.)-- (0.,-2.);
\draw [dash pattern=on 1pt off 1pt,color=ffqqqq] (0.755994,0.)-- (0.,-2.);
\draw [dash pattern=on 1pt off 1pt,color=ffqqqq] (-1.5542442734793633,1.2586996219736486)-- (1.55685,1.25547);
\draw [dash pattern=on 1pt off 1pt,color=ffqqqq] (-1.5542442734793633,1.2586996219736486)-- (-0.755994,0.);
\begin{scriptsize}
\draw [fill=ffqqqq] (0.,-2.) circle (0.5pt);
\draw [fill=ffqqqq] (0.,1.74505) circle (1.5pt);
\draw [fill=ffqqqq] (-1.41505,-0.370369) circle (1.5pt);
\draw [fill=ffqqqq] (0.,0.) circle (1.5pt);
\draw [fill=ffqqqq] (1.41505,-0.370369) circle (1.5pt);
\draw [fill=ffqqqq] (-1.55685,1.25547) circle (0.5pt);
\draw [fill=ffqqqq] (-0.755994,0.) circle (0.5pt);
\draw [fill=ffqqqq] (1.55685,1.25547) circle (0.5pt);
\draw [fill=ffqqqq] (0.755994,0.) circle (0.5pt);
\draw[] (0.00096496167399,-2.330438200585347174) node { $(i)$ };
\end{scriptsize}
\end{tikzpicture}
\begin{tikzpicture}[line cap=round,line join=round,>=triangle 45,x=1.0cm,y=1.0cm]
\clip(-2.1431409553721457,-2.454073612323069) rectangle (2.119731706345336,2.1580023224116287);
\draw(0.,0.) circle (2.cm);
\draw (-1.9739043510998817,0.)-- (2.0260956489001183,0.);
\draw [dash pattern=on 1pt off 1pt,color=ffqqqq] (0.,-2.)-- (0.,2.);
\draw (0.00043764394213067,2.109114050825449) node[anchor=north west] {$p_1$};
\draw (-1.3125420212039725,0.4868602805213256) node[anchor=north west] {$p_2$};
\draw (0.000539167168129362,-1.3888706413928176) node[anchor=north west] {$p_3$};
\draw (1.2391279716702248,0.46996180374732427) node[anchor=north west] {$p_4$};
\draw (-1.8612069870060027,1.8894338527634325) node[anchor=north west] {$d_1$};
\draw [dash pattern=on 1pt off 1pt,color=ffqqqq] (-1.39228,1.43581)-- (0.,0.);
\draw [dash pattern=on 1pt off 1pt,color=ffqqqq] (0.,0.)-- (1.39228,1.43581);
\draw [dash pattern=on 1pt off 1pt,color=ffqqqq] (0.,0.)-- (1.39228,-1.43581);
\draw [dash pattern=on 1pt off 1pt,color=ffqqqq] (0.,0.)-- (-1.39228,-1.43581);
\begin{scriptsize}
\draw [fill=ffqqqq] (0.,-2.) circle (0.5pt);
\draw [fill=ffqqqq] (0.,1.80814) circle (1.5pt);
\draw [fill=ffqqqq] (-1.26776,0.) circle (1.5pt);
\draw [fill=ffqqqq] (0.,-1.80814) circle (1.5pt);
\draw [fill=ffqqqq] (1.26776,0.) circle (1.5pt);
\draw [fill=ffqqqq] (-1.39228,1.43581) circle (0.5pt);
\draw [fill=ffqqqq] (-1.39228,-1.43581) circle (0.5pt);
\draw [fill=ffqqqq] (1.39228,1.43581) circle (0.5pt);
\draw [fill=ffqqqq] (1.39228,-1.43581) circle (0.5pt);
\draw[] (0.00096496167399,-2.330438200585347174) node { $(ii)$ };
\end{scriptsize}
\end{tikzpicture}
\caption{ } \label{Fig01}
\end{figure}

\begin{prop} \label{prop5}
An optimal set of five-means is given by \[\set{(0, 0.903584), (-0.788308, 0), (0, 0), (0, -0.903584), (0.788308, 0)}\] and the corresponding quantization error is $V_5=0.119779$.
\end{prop}

\begin{proof} Let $\ga:=\set{p_1, p_2, p_3, p_4, p_5}$ be an optimal set of five-means. The following cases can arise:

\tit{Case~1. $\ga$ contains two points from $L_1$, the Voronoi regions of which do not contain any point from $L_2$.}

In this case, we can assume that $p_1, p_2, \cdots, p_5$ can be located as shown in Figure~\ref{Fig11}~$(i)$. Let the boundary of the Voronoi regions of $p_1$ and $p_2$ intersect $L_2$ at the point $d_1$ given by the parametric value $\gq=\pi-b$, where $0<b<\frac {\pi}2$, and the boundary of the Voronoi regions of $p_2$ and $p_3$ intersect $L_1$ at the point $d_2$ given by $x_1=-a$, where $0<a<1$. Thus, due to symmetry, we have
\begin{align*} p_1&=\frac{\int_b^{\pi -b} (\cos \theta ,\sin \theta )\, d\theta }{\int_b^{\pi -b}   d\theta }=\Big(0,\frac{2 \cos b}{\pi -2 b}\Big),\\
p_2&=\frac{\frac{1}{4} \int_{-1}^{-a} (x,0) \, dx+\frac1{4\pi} \int_{\pi -b}^{\frac {3\pi}2} (\cos \theta,\sin \theta) \, d\theta}{\frac{1}{4} \int_{-1}^{-a} dx+\frac1{4\pi} \int_{\pi -b}^{\frac {3\pi}2}  d\theta}=\Big(\frac{-\pi  a^2+2 \sin b+\pi +2}{\pi  (2 a-3)-2 b},-\frac{2 \cos b}{-2 \pi  a+2 b+3 \pi }\Big), \\
p_3&=(-\frac a 2, 0), \quad d_1=(-\cos b, \sin b), \te{ and } d_2=(-a, 0).
\end{align*}
Thus, solving the canonical equations $\rho(d_1, p_1)-\rho(d_1, p_2)=0$, and $\rho(d_2, p_2)-\rho(d_2, p_3)=0$, we have $a=0.567815$, $b=0.656426$. Hence, putting the values of $a$ and $b$ we have, $p_1=(0, 0.866365)$, $p_2=(-0.74607, -0.220972)$, and $p_3=(-0.283907, 0)$, and so, due to symmetry, $p_4=(0.283907, 0)$, and $p_5=(0.74607, -0.220972)$. The corresponding distortion error is given by
\begin{align*}
V(P; \ga)&=\frac 1{4\pi} \int_b^{\pi -b} \rho((\cos \theta ,\sin \theta ), p_1)\, d\theta+2\Big(\frac{1}{4} \int_{-1}^{-a} \rho((x,0), p_2) \, dx\\
&\qquad  +\frac1{4\pi} \int_{\pi -b}^{\frac {3\pi}2} \rho((\cos \theta,\sin \theta), p_2) \, d\theta
+\frac{1}{4} \int_{-a}^{0} \rho((x,0), p_3) \, dx\Big)=0.18911.
\end{align*}

\tit{Case~2. $\ga$ contains only one point from $L_1$, the Voronoi region of which does not contain any point from $L_2$.}

In this case, we can assume that $p_1, p_2, \cdots, p_5$ can be located as shown in Figure~\ref{Fig11}~$(ii)$. Let the boundary of the Voronoi regions of $p_1$ and $p_2$ intersect $L_2$ at the point $d_1$ given by the parametric value $\gq=\pi-b$, where $0<b<\frac {\pi}2$, the boundary of the Vonoroi regions of $p_2$ and $p_3$ intersect $L_1$ at the point $d_2$ given by $x_1=-a$, where $0<a<1$. Thus, due to symmetry, we have
\begin{align*} p_1&=\frac{\int_{b}^{\pi -b} (\cos \theta ,\sin \theta )\, d\theta }{\int_{b}^{\pi -b}   d\theta }=\Big(0,\frac{2 \cos b}{\pi -2 b}\Big),\\
p_2&=\frac{\frac{1}{4} \int_{-1}^{-a} (x,0) \, dx+\frac1{4\pi} \int_{\pi -b}^{\pi+b} (\cos \theta,\sin \theta) \, d\theta}{\frac{1}{4} \int_{-1}^{-a} dx+\frac1{4\pi} \int_{\pi -b}^{\pi+b}  d\theta}=\Big(-\frac{-\pi  a^2+4 \sin b+\pi }{-2 \pi  a+4 b+2 \pi },0\Big), \\
p_3&=(0, 0),  \quad d_1=(-\cos b, \sin b), \quad  d_2=(-a, 0).
\end{align*}
Thus, solving the canonical equations $\rho(d_1, p_1)-\rho(d_1, p_2)=0$, $\rho(d_2, p_2)-\rho(d_2, p_3)=0$, we have $a=0.394154$, and $b=0.798783$. Hence, putting the values of $a$, and $b$, we have, $p_1=(0, 0.903584)$, $p_2=(-0.788308, 0)$, and $p_3=(0, 0)$, and so, due to symmetry, $p_4=(0, -0.903584)$, and $p_5=(0.788308, 0)$. The corresponding distortion error is given by
\begin{align*}
V(P; \ga)&=2\Big(\frac 1{4\pi} \int_b^{\pi -b} \rho((\cos \theta ,\sin \theta ), p_1)\, d\theta +\frac{1}{4} \int_{-1}^{-a} \rho((x,0), p_2) \, dx\\
&\qquad+\frac1{4\pi} \int_{\pi -b}^{\pi+b} \rho((\cos \theta,\sin \theta),p_2) \, d\theta\Big)+\frac{1}{4} \int_{-a}^{a} \rho((x,0), p_3) \, dx=0.119779.
\end{align*}

\tit{Case~3. $\ga$ does not contain any point from $L_1$, the Voronoi region of which does not contain any point from $L_2$.}

In this case, we can assume that $p_1, p_2, \cdots, p_5$ can be located as shown in Figure~\ref{Fig11}~$(iii)$. Let the boundary of the Voronoi regions of $p_1$ and $p_2$ intersect $L_2$ at the point $d_1$ given by the parametric value $\gq=\pi-b$, where $0<b<\frac {\pi}2$, and the boundary of the Voronoi regions of $p_2$ and $p_3$ intersect $L_2$ as the point $d_2$ given by the parametric value $\gq=\pi+c$, where $0<c<\frac {\pi}2$. Thus, due to symmetry, we have
\begin{align*} p_1&=\frac{\int_{\frac \pi 2}^{\pi -b} (\cos \theta ,\sin \theta )\, d\theta }{\int_{\frac \pi 2}^{\pi -b}   d\theta }=\Big(\frac{2 (\sin b-1)}{\pi -2 b},\frac{2 \cos b}{\pi -2 b}\Big),\\
p_2&=\frac{\frac{1}{4} \int_{-1}^{0} (x,0) \, dx+\frac1{4\pi} \int_{\pi -b}^{\pi+c} (\cos \theta,\sin \theta) \, d\theta}{\frac{1}{4} \int_{-1}^{0} dx+\frac1{4\pi} \int_{\pi -b}^{\pi+c}  d\theta}=\Big(-\frac{2 \sin b+2 \sin c+\pi }{2 (b+c+\pi )},\frac{\cos c-\cos b}{b+c+\pi }\Big), \\
p_3&=\frac{\int_{\pi+c}^{2\pi -c} (\cos \theta ,\sin \theta )\, d\theta }{\int_{\pi+c}^{2\pi -c}   d\theta }=\Big(0,-\frac{2 \cos c}{\pi -2 c}\Big)\\
d_1&=(-\cos b, \sin b), \te{ and } d_2=(-\cos c, -\sin c).
\end{align*}
Thus, solving the canonical equations $\rho(d_1, p_1)-\rho(d_1, p_2)=0$, and $\rho(d_2, p_2)-\rho(d_2, p_3)=0$, we have $b=0.426473$, and $c=0.837847$. Hence, putting the values of $b$, and $c$, we have, $p_1=(-0.512388, 0.795606)$, $p_2=(-0.619091, -0.0547824)$, $p_3=(0, -0.912839)$, and so, due to symmetry, $p_4=(0.619091, -0.0547824)$, and $p_5=(0.512388, 0.795606)$. The corresponding distortion error is given by
\begin{align*}
V(P; \ga)&=2\Big(\frac 1{4\pi}\int_{\frac \pi 2}^{\pi -b} \rho((\cos \theta ,\sin \theta ), p_1)\, d\theta \Big)+\frac{1}{4} \int_{-1}^{0} \rho((x,0), p_2) \, dx\\
&\qquad  +\frac1{4\pi} \int_{\pi -b}^{\pi+c} \rho((\cos \theta,\sin \theta), p_2) \, d\theta \Big)+\frac 1{4\pi}\int_{\pi+c}^{2\pi -c} \rho((\cos \theta ,\sin \theta ), p_3)\, d\theta=0.1355.
\end{align*}
Comparing Case~1 and Case~2, we see that if $\ga$ contains two points from $L_1$, the Voronoi regions of which do not contain any point from $L_2$, then the distortion error is larger than the distortion error obtained in Case~2. Similarly, we can show that if $\ga$ contains more than two points from $L_1$, the Voronoi regions of which do not contain any point from $L_2$, then the distortion error is larger than the distortion error obtained in Case~2. Comparing  Case~2 and Case~3, we see that Case~3 can not happen as the distortion error is larger in Case~3. Considering all the above cases, we see that the distortion error in Case~2 is the smallest. Hence, the points in $\ga$ obtained in Case~2 form an optimal set of five-means, and the corresponding quantization error is given by $V_5=0.119779$.
Thus, the proof of the proposition is complete.
 \end{proof}

\begin{figure}
\begin{tikzpicture}[line cap=round,line join=round,>=triangle 45,x=1.0cm,y=1.0cm]
\clip(-2.1431409553721457,-2.454073612323069) rectangle (2.119731706345336,2.1580023224116287);
\draw(0.,0.) circle (2.cm);
\draw (-1.9739043510998817,0.)-- (2.0260956489001183,0.);
\draw [dash pattern=on 1pt off 1pt,color=ffqqqq] (0.,-2.)-- (0.,2.);
\draw [dash pattern=on 1pt off 1pt,color=ffqqqq] (1.58436,1.22058)-- (1.13563,0.);
\draw [dash pattern=on 1pt off 1pt,color=ffqqqq] (-1.13563,0.)-- (0.,-2.);
\draw [dash pattern=on 1pt off 1pt,color=ffqqqq] (1.13563,0.)-- (0.,-2.);
\draw [dash pattern=on 1pt off 1pt,color=ffqqqq] (-1.5827931763841139,1.2226061347759907)-- (1.58436,1.22058);
\draw [dash pattern=on 1pt off 1pt,color=ffqqqq] (-1.5827931763841139,1.2226061347759907)-- (-1.13563,0.);
\draw (0.000336120716131975,2.109114050825449) node[anchor=north west] {$p_1$};
\draw (-1.5615267889439856,-0.070789453020717) node[anchor=north west] {$p_2$};
\draw (-0.6866029502439203,0.5206572340693282) node[anchor=north west] {$p_3$};
\draw (0.4462904239361712,0.5206572340693282) node[anchor=north west] {$p_4$};
\draw (1.4326051232802444,-0.10458640656871958) node[anchor=north west] {$p_5$};
\draw (-2.0708871850680196,1.6697536547014158) node[anchor=north west] {$d_1$};
\draw (-1.287354207011962,0.4868602805213256) node[anchor=north west] {$d_2$};
\begin{scriptsize}
\draw [fill=ffqqqq] (0.,1.73273) circle (1.5pt);
\draw [fill=ffqqqq] (-1.49214,-0.441944) circle (1.5pt);
\draw [fill=ffqqqq] (-0.567815,0.) circle (1.5pt);
\draw [fill=ffqqqq] (1.49214,-0.441944) circle (1.5pt);
\draw [fill=ffqqqq] (0.567815,0.) circle (1.5pt);
\draw [fill=ffqqqq] (-1.58436,1.22058) circle (0.5pt);
\draw [fill=ffqqqq] (-1.13563,0.) circle (0.5pt);
\draw [fill=ffqqqq] (1.58436,1.22058) circle (0.5pt);
\draw [fill=ffqqqq] (1.13563,0.) circle (0.5pt);
\draw [fill=ffqqqq] (0.,-2.) circle (0.5pt);
\draw[] (0.00096496167399,-2.330438200585347174) node { $(i)$ };
\end{scriptsize}
\end{tikzpicture}
\begin{tikzpicture}[line cap=round,line join=round,>=triangle 45,x=1.0cm,y=1.0cm]
\clip(-2.1431409553721457,-2.454073612323069) rectangle (2.119731706345336,2.1580023224116287);
\draw(0.,0.) circle (2.cm);
\draw (-1.9739043510998817,0.)-- (2.0260956489001183,0.);
\draw [dash pattern=on 1pt off 1pt,color=ffqqqq] (0.,-2.)-- (0.,2.);
\draw (0.00000004587222234,2.11986762695436) node[anchor=north west] {$p_1$};
\draw (-1.6305406295874518,0.46074445277968734) node[anchor=north west] {$p_2$};
\draw (0.00000004587222234,0.43001994955423045) node[anchor=north west] {$p_3$};
\draw (0.00000004587222234,-1.4902615020368257) node[anchor=north west] {$p_4$};
\draw (1.5033586994091548,0.46074445277968734) node[anchor=north west] {$p_5$};
\draw (-1.9012651328129089,1.8433470979252478) node[anchor=north west] {$d_1$};
\draw (-0.8970657973383002,0.46074445277968734) node[anchor=north west] {$d_2$};
\draw [dash pattern=on 1pt off 1pt,color=ffqqqq] (1.39516,1.43302)-- (0.788308,0.);
\draw [dash pattern=on 1pt off 1pt,color=ffqqqq] (0.788308,0.)-- (1.39516,-1.43302);
\draw [dash pattern=on 1pt off 1pt,color=ffqqqq] (1.39516,-1.43302)-- (-1.39516,-1.43302);
\draw [dash pattern=on 1pt off 1pt,color=ffqqqq] (-1.39516,-1.43302)-- (-0.788308,0.);
\draw [dash pattern=on 1pt off 1pt,color=ffqqqq] (1.39516,1.43302)-- (-1.39516,1.43302);
\draw [dash pattern=on 1pt off 1pt,color=ffqqqq] (-1.39516,1.43302)-- (-0.788308,0.);
\begin{scriptsize}
\draw [fill=ffqqqq] (0.,1.80717) circle (1.5pt);
\draw [fill=ffqqqq] (-1.57662,0.) circle (1.5pt);
\draw [fill=ffqqqq] (0.,0.) circle (1.5pt);
\draw [fill=ffqqqq] (0.,-1.80717) circle (1.5pt);
\draw [fill=ffqqqq] (1.57662,0.) circle (1.5pt);
\draw [fill=ffxfqq] (-1.39516,1.43302) circle (0.5pt);
\draw [fill=ffxfqq] (-0.788308,0.) circle (0.5pt);
\draw [fill=ffxfqq] (-1.39516,-1.43302) circle (0.5pt);
\draw [fill=ffxfqq] (-0.788308,0.) circle (0.5pt);
\draw [fill=ffxfqq] (1.39516,1.43302) circle (0.5pt);
\draw [fill=ffxfqq] (1.39516,-1.43302) circle (0.5pt);
\draw [fill=ffqqqq] (0.788308,0.) circle (0.5pt);
\draw[] (0.00096496167399,-2.330438200585347174) node { $(ii)$ };
\end{scriptsize}
\end{tikzpicture}
\begin{tikzpicture}[line cap=round,line join=round,>=triangle 45,x=1.0cm,y=1.0cm]
\clip(-2.1431409553721457,-2.454073612323069) rectangle (2.119731706345336,2.1580023224116287);
\draw(0.,0.) circle (2.cm);
\draw (-1.9739043510998817,0.)-- (2.0260956489001183,0.);
\draw [dash pattern=on 1pt off 1pt,color=ffqqqq] (0.,-2.)-- (0.,2.);
\draw (-1.0477840221108924,1.9459669386982732) node[anchor=north west] {$p_1$};
\draw (-1.2635800479145478,0.17930800323450138) node[anchor=north west] {$p_2$};
\draw (0.00000000716457705,-1.4490906677147144) node[anchor=north west] {$p_3$};
\draw (1.2097424617347348,0.16394575162177294) node[anchor=north west] {$p_4$};
\draw (1.0253954423819933,1.7923444225709886) node[anchor=north west] {$p_5$};
\draw (-2.3099528930655274,1.253216609674579) node[anchor=north west] {$d_1$};
\draw (-1.7583608414582165,-1.2186568935237876) node[anchor=north west] {$d_2$};
\draw [dash pattern=on 1pt off 1pt,color=ffqqqq] (0.,0.)-- (-1.33813,-1.48641);
\draw [dash pattern=on 1pt off 1pt,color=ffqqqq] (-1.33813,-1.48641)-- (1.345674764197645,-1.4795808288165984);
\draw [dash pattern=on 1pt off 1pt,color=ffqqqq] (1.345674764197645,-1.4795808288165984)-- (0.,0.);
\draw [dash pattern=on 1pt off 1pt,color=ffqqqq] (0.,0.)-- (1.82086,0.827324);
\draw [dash pattern=on 1pt off 1pt,color=ffqqqq] (1.82086,0.827324)-- (0.,2.);
\draw [dash pattern=on 1pt off 1pt,color=ffqqqq] (0.,2.)-- (-1.82086,0.827324);
\draw [dash pattern=on 1pt off 1pt,color=ffqqqq] (-1.82086,0.827324)-- (0.,0.);
\begin{scriptsize}
\draw [fill=ffqqqq] (-1.02478,1.59121) circle (1.5pt);
\draw [fill=ffqqqq] (-1.23818,-0.109565) circle (1.5pt);
\draw [fill=ffqqqq] (0.,-1.82568) circle (1.5pt);
\draw [fill=ffqqqq] (1.02478,1.59121) circle (1.5pt);
\draw [fill=ffqqqq] (1.23818,-0.109565) circle (1.5pt);
\draw [fill=ffqqqq] (-1.82086,0.827324) circle (0.5pt);
\draw [fill=ffqqqq] (-1.33813,-1.48641) circle (0.5pt);
\draw [fill=ffqqqq] (1.82086,0.827324) circle (0.5pt);
\draw [fill=ffqqqq] (1.33813,-1.48641) circle (0.5pt);
\draw [fill=ffqqqq] (0.,2.) circle (0.5pt);
\draw [fill=ffqqqq] (0.,0.) circle (0.5pt);
\draw[] (0.00096496167399,-2.330438200585347174) node { $(iii)$ };
\end{scriptsize}
\end{tikzpicture}
\caption{ } \label{Fig11}
\end{figure}

\begin{prop} \label{prop6}
An optimal set of six-means is
 \begin{align*}
\set{(-0.497577,0.809422), & (-0.786245,-0.0706781),(0, 0), (0,-0.913921), (0.786245,-0.0706781), \\
& (0.497577,0.809422)}
 \end{align*} and the corresponding quantization error for six-means is given by $V_6=0.093342$.
\end{prop}
\begin{proof} Let $\ga:=\set{p_1, p_2, p_3, p_4, p_5, p_6}$ be an optimal set of six-means. As in Proposition~\ref{prop5}, here also we consider three different cases as shown in Figure~\ref{Fig2}. In each case, we calculate the distortion errors. Then, comparing the distortion errors, we see that the points given by the proposition give the smallest distortion error for six points, and hence they form an optimal set of six-means, which is shown by Figure~\ref{Fig2}~$(ii)$.  Thus, the proof of the proposition is deduced.
\begin{figure}
\begin{tikzpicture}[line cap=round,line join=round,>=triangle 45,x=1.0cm,y=1.0cm]
\clip(-2.1431409553721457,-2.454073612323069) rectangle (2.119731706345336,2.1580023224116287);
\draw(0.,0.) circle (2.cm);
\draw (-1.9739043510998817,0.)-- (2.0260956489001183,0.);
\draw [dash pattern=on 1pt off 1pt,color=ffqqqq] (0.,-2.)-- (0.,2.);
\draw [dash pattern=on 1pt off 1pt,color=ffqqqq] (-1.40311,1.42523)-- (-1.12543,0.);
\draw [dash pattern=on 1pt off 1pt,color=ffqqqq] (-1.12543,0.)-- (-1.40311,-1.42523);
\draw [dash pattern=on 1pt off 1pt,color=ffqqqq] (-1.40311,-1.42523)-- (1.3907039633856693,-1.437338681808637);
\draw [dash pattern=on 1pt off 1pt,color=ffqqqq] (1.40311,1.42523)-- (1.12543,0.);
\draw [dash pattern=on 1pt off 1pt,color=ffqqqq] (1.12543,0.)-- (1.3907039633856693,-1.437338681808637);
\draw [dash pattern=on 1pt off 1pt,color=ffqqqq] (1.40311,1.42523)-- (-1.40311,1.42523);
\draw (0.00003155103813541,2.092215574051448) node[anchor=north west] {$p_1$};
\draw (-1.9519024173280071,0.5375557108433294) node[anchor=north west] {$p_2$};
\draw (-0.6366029502439208,0.5375557108433294) node[anchor=north west] {$p_3$};
\draw (-1.8012069870060031,1.9401292830854364) node[anchor=north west] {$d_1$};
\draw (-1.3604557302379612,0.5375557108433294) node[anchor=north west] {$d_2$};
\draw (0.42179804006616415,0.5375557108433294) node[anchor=north west] {$p_4$};
\draw (0.0093361207161315,-1.4057691181668188) node[anchor=north west] {$p_5$};
\draw (1.5101990303762492,0.5375557108433294) node[anchor=north west] {$p_6$};
\begin{scriptsize}
\draw [fill=ffqqqq] (0.,1.80446) circle (1.5pt);
\draw [fill=ffqqqq] (-1.68815,0.) circle (1.5pt);
\draw [fill=ffqqqq] (-0.562717,0.) circle (1.5pt);
\draw [fill=ffqqqq] (0.,-1.80446) circle (1.5pt);
\draw [fill=ffqqqq] (1.68815,0.) circle (1.5pt);
\draw [fill=ffqqqq] (0.562717,0.) circle (1.5pt);
\draw [fill=ffqqqq] (-1.40311,1.42523) circle (0.5pt);
\draw [fill=ffqqqq] (-1.12543,0.) circle (0.5pt);
\draw [fill=ffqqqq] (1.12543,0.) circle (0.5pt);
\draw [fill=ffqqqq] (-1.40311,-1.42523) circle (0.5pt);
\draw [fill=ffqqqq] (1.40311,1.42523) circle (0.5pt);
\draw [fill=ffqqqq] (1.40311,-1.42523) circle (0.5pt);
\draw[] (0.00096496167399,-2.330438200585347174) node { $(i)$ };
\end{scriptsize}
\end{tikzpicture}
\begin{tikzpicture}[line cap=round,line join=round,>=triangle 45,x=1.0cm,y=1.0cm]
\clip(-2.1431409553721457,-2.454073612323069) rectangle (2.119731706345336,2.1580023224116287);
\draw(0.,0.) circle (2.cm);
\draw (-1.9739043510998817,0.)-- (2.0260956489001183,0.);
\draw [dash pattern=on 1pt off 1pt,color=ffqqqq] (0.,-2.)-- (0.,2.);
\draw [dash pattern=on 1pt off 1pt,color=ffqqqq] (-1.78455,0.902981)-- (-0.792598,0.);
\draw [dash pattern=on 1pt off 1pt,color=ffqqqq] (-0.792598,0.)-- (-1.3195429979402469,-1.5029325588950642);
\draw [dash pattern=on 1pt off 1pt,color=ffqqqq] (1.3088544602241603,-1.512249979980599)-- (0.792598,0.);
\draw [dash pattern=on 1pt off 1pt,color=ffqqqq] (0.792598,0.)-- (1.78455,0.902981);
\draw [dash pattern=on 1pt off 1pt,color=ffqqqq] (-1.78455,0.902981)-- (1.78455,0.902981);
\draw [dash pattern=on 1pt off 1pt,color=ffqqqq] (-1.3195429979402469,-1.5029325588950642)-- (1.33114,-1.49267);
\draw (-1.1076740089499456,1.7880429921194247) node[anchor=north west] {$p_1$};
\draw (-1.682917649587994,0.00019531471929589) node[anchor=north west] {$p_2$};
\draw (0.00043764394213019,0.4037587572953268) node[anchor=north west] {$p_3$};
\draw (0.000039167168128883,-1.4888706413928176) node[anchor=north west] {$p_4$};
\draw (0.9011584361901981,1.7880429921194247) node[anchor=north west] {$p_5$};
\draw (1.4601990303762492,0.00019531471929589) node[anchor=north west] {$p_6$};
\draw (-2.2067704295820344,1.4486825959953912) node[anchor=north west] {$d_1$};
\draw (-0.9469785786279416,0.5375557108433294) node[anchor=north west] {$d_2$};
\draw (-1.7491206960399913,-1.255073687844815) node[anchor=north west] {$d_3$};
\draw [dash pattern=on 1pt off 1pt,color=ffqqqq] (0.,2.)-- (-1.78455,0.902981);
\draw [dash pattern=on 1pt off 1pt,color=ffqqqq] (0.,2.)-- (1.7845512681823221,0.902982154435438);
\begin{scriptsize}
\draw [fill=ffqqqq] (-0.995154,1.61884) circle (1.5pt);
\draw [fill=ffqqqq] (-1.57249,-0.141356) circle (1.5pt);
\draw [fill=ffqqqq] (0.,0.) circle (1.5pt);
\draw [fill=ffqqqq] (0.,-1.82784) circle (1.5pt);
\draw [fill=ffqqqq] (0.995154,1.61884) circle (1.5pt);
\draw [fill=ffqqqq] (1.57249,-0.141356) circle (1.5pt);
\draw [fill=ffqqqq] (-1.78455,0.902981) circle (0.5pt);
\draw [fill=ffqqqq] (-0.792598,0.) circle (0.5pt);
\draw [fill=ffqqqq] (-1.33114,-1.49267) circle (0.5pt);
\draw [fill=ffqqqq] (1.78455,0.902981) circle (0.5pt);
\draw [fill=ffqqqq] (0.792598,0.) circle (0.5pt);
\draw [fill=ffqqqq] (1.33114,-1.49267) circle (0.5pt);
\draw[] (0.00096496167399,-2.330438200585347174) node { $(ii)$ };
\end{scriptsize}
\end{tikzpicture}
\begin{tikzpicture}[line cap=round,line join=round,>=triangle 45,x=1.0cm,y=1.0cm]
\clip(-2.1431409553721457,-2.454073612323069) rectangle (2.119731706345336,2.1580023224116287);
\draw(0.,0.) circle (2.cm);
\draw (-1.9739043510998817,0.)-- (2.0260956489001183,0.);
\draw [dash pattern=on 1pt off 1pt,color=ffqqqq] (0.,-2.)-- (0.,2.);
\draw (-1.042166392819952,1.8218399456674275) node[anchor=north west] {$p_1$};
\draw (-1.2618465908819692,0.5375557108433294) node[anchor=north west] {$p_2$};
\draw (-1.1183694392719496,-1.118495013008797) node[anchor=north west] {$p_3$};
\draw (0.8335645290941929,-1.118495013008797) node[anchor=north west] {$p_4$};
\draw (1.2053310181222217,0.5037587572953268) node[anchor=north west] {$p_5$};
\draw (1.002549296834206,1.8218399456674275) node[anchor=north west] {$p_6$};
\draw (-2.2667704295820344,1.3886825959953912) node[anchor=north west] {$d_1$};
\draw [dash pattern=on 1pt off 1pt,color=ffqqqq] (-1.7819682440770641,0.9080689275054534)-- (0.,0.);
\draw [dash pattern=on 1pt off 1pt,color=ffqqqq] (0.,0.)-- (-1.78606,-0.899997);
\draw [dash pattern=on 1pt off 1pt,color=ffqqqq] (0.,0.)-- (1.78606,0.899997);
\draw [dash pattern=on 1pt off 1pt,color=ffqqqq] (0.,0.)-- (1.78606,-0.899997);
\draw [dash pattern=on 1pt off 1pt,color=ffqqqq] (-1.78606,0.899997)-- (0.,2.);
\draw [dash pattern=on 1pt off 1pt,color=ffqqqq] (-1.78606,-0.899997)-- (0.,-2.);
\draw [dash pattern=on 1pt off 1pt,color=ffqqqq] (0.,-2.)-- (1.78606,-0.899997);
\draw [dash pattern=on 1pt off 1pt,color=ffqqqq] (0.,2.)-- (1.78606,0.899997);
\begin{scriptsize}
\draw [fill=ffqqqq] (-0.99635,1.61776) circle (1.5pt);
\draw [fill=ffqqqq] (-1.21262,0.) circle (1.5pt);
\draw [fill=ffqqqq] (-0.99635,-1.61776) circle (1.5pt);
\draw [fill=ffqqqq] (1.21262,0.) circle (1.5pt);
\draw [fill=ffqqqq] (0.99635,-1.61776) circle (1.5pt);
\draw [fill=ffqqqq] (0.99635,1.61776) circle (1.5pt);
\draw [fill=ffqqqq] (-1.78606,0.899997) circle (0.5pt);
\draw [fill=ffqqqq] (-1.78606,-0.899997) circle (0.5pt);
\draw [fill=ffqqqq] (1.78606,0.899997) circle (0.5pt);
\draw [fill=ffqqqq] (1.78606,-0.899997) circle (0.5pt);
\draw[] (0.00096496167399,-2.330438200585347174) node { $(iii)$ };
\end{scriptsize}
\end{tikzpicture}
\caption{ } \label{Fig2}
\end{figure}
\end{proof}

Proceeding in the similar way as Proposition~\ref{prop5} and Proposition~\ref{prop6}, we can deduce that the following proposition is also true.

 \begin{prop} \label{prop7} Let $\ga_n$ be an optimal set of $n$-means, and let $V_n$ be the corresponding quantization error. Then,
 \begin{align*}
\ga_7=\set{&(-0.476891, 0.827476), (-0.788772, 0), (0, 0), (-0.476891, -0.827476), \\
& (0.476891, -0.827476),  (0.788772, 0), (0.476891, 0.827476)},
 \end{align*} with $V_7=0.070674$, see Figure~\ref{Fig3}~$(i)$;
 \begin{align*}
\ga_8=\set{& (-0.475258,0.828843), (-0.860649,0),(-0.286883,0), (-0.475258, -0.828843), \\
&(0.475258, -0.828843), (0.860649,0),(0.286883,0), (0.475258, 0.828843)},
 \end{align*} with $V_8=0.0577852$, see Figure~\ref{Fig3}~$(ii)$;
  \begin{align*}
\ga_9=\set{& (-0.463928,0.838108),(-0.857223,0.0396484),(-0.286659,0),(-0.704114,-0.671446),\\
&   (0,-0.972943),(0.704114,-0.671446),(0.286659,0), (0.857223,0.0396484),\\
&   (0.463928,0.838108)},
 \end{align*}  with $V_9=0.04803$, see Figure~\ref{Fig3}~$(iii)$;
 \begin{align*}
\ga_{10}=\set{& (0,0.974386),(-0.690161,0.687826),(-0.854308,0),(-0.284769,0),\\
&  (-0.690161, -0.687826),(0,-0.974386),(0.690161, -0.687826), (0.854308,0), \\
& (0.284769,0),  (0.690161,0.687826)},
 \end{align*} with $V_{10}=0.039046$,  see Figure~\ref{Fig3}~$(iv)$.
 \end{prop}

\begin{figure}
\begin{tikzpicture}[line cap=round,line join=round,>=triangle 45,x=1.0cm,y=1.0cm]
\clip(-2.1431409553721457,-2.454073612323069) rectangle (2.119731706345336,2.1580023224116287);
\draw(0.,0.) circle (2.cm);
\draw (-1.9739043510998817,0.)-- (2.0260956489001183,0.);
\draw [dash pattern=on 1pt off 1pt,color=ffqqqq] (0.,-2.)-- (0.,2.);
\draw (-0.9301955940961008,2.1023137975970934) node[anchor=north west] {$p_1$};
\draw (-1.6037794766555944,0.5133466900208501) node[anchor=north west] {$p_2$};
\draw (0.00000157337005516,0.3888187291676852) node[anchor=north west] {$p_3$};
\draw (-0.9129242124920112,-1.1792487071799311) node[anchor=north west] {$p_4$};
\draw (0.9351136191455742,-1.1792487071799311) node[anchor=north west] {$p_5$};
\draw (1.5050692120805305,0.5133466900208501) node[anchor=north west] {$p_6$};
\draw (0.9523850007496638,1.912328599952108) node[anchor=north west] {$p_7$};
\draw [dash pattern=on 1pt off 1pt,color=ffqqqq] (0.,2.)-- (1.7305,1.00268);
\draw [dash pattern=on 1pt off 1pt,color=ffqqqq] (1.7305,1.00268)-- (0.788772,0.);
\draw [dash pattern=on 1pt off 1pt,color=ffqqqq] (0.788772,0.)-- (1.7305,-1.00268);
\draw [dash pattern=on 1pt off 1pt,color=ffqqqq] (0.,-2.)-- (1.7305,-1.00268);
\draw [dash pattern=on 1pt off 1pt,color=ffqqqq] (0.,-2.)-- (-1.7305,-1.00268);
\draw [dash pattern=on 1pt off 1pt,color=ffqqqq] (-1.7305,-1.00268)-- (-0.7747531596592945,0.);
\draw [dash pattern=on 1pt off 1pt,color=ffqqqq] (-0.7747531596592945,0.)-- (-1.7305,1.00268);
\draw [dash pattern=on 1pt off 1pt,color=ffqqqq] (-1.7305,1.00268)-- (0.,2.);
\begin{scriptsize}
\draw [fill=ffqqqq] (0.,-2.) circle (0.5pt);
\draw [fill=ffqqqq] (-0.953782,1.65495) circle (1.5pt);
\draw [fill=ffqqqq] (-1.57754,0.) circle (1.5pt);
\draw [fill=ffqqqq] (0.,0.) circle (1.5pt);
\draw [fill=ffqqqq] (0.953782,1.65495) circle (1.5pt);
\draw [fill=ffqqqq] (1.57754,0.) circle (1.5pt);
\draw [fill=ffqqqq] (-0.953782,-1.65495) circle (1.5pt);
\draw [fill=ffqqqq] (0.953782,-1.65495) circle (1.5pt);
\draw [fill=ffqqqq] (0.,2.) circle (0.5pt);
\draw [fill=ffqqqq] (-1.7305,1.00268) circle (0.5pt);
\draw [fill=ffqqqq] (-0.788772,0.) circle (0.5pt);
\draw [fill=ffqqqq] (-1.7305,-1.00268) circle (0.5pt);
\draw [fill=ffqqqq] (0.788772,0.) circle (0.5pt);
\draw [fill=ffqqqq] (1.7305,1.00268) circle (0.5pt);
\draw [fill=ffqqqq] (-1.7305,-1.00268) circle (0.5pt);
\draw [fill=ffqqqq] (1.7305,-1.00268) circle (0.5pt);
\draw[] (0.00096496167399,-2.330438200585347174) node { $(i)$ };
\end{scriptsize}
\end{tikzpicture}
\begin{tikzpicture}[line cap=round,line join=round,>=triangle 45,x=1.0cm,y=1.0cm]
\clip(-2.1431409553721457,-2.454073612323069) rectangle (2.119731706345336,2.1580023224116287);
\draw(0.,0.) circle (2.cm);
\draw (-1.9739043510998817,0.)-- (2.0260956489001183,0.);
\draw [dash pattern=on 1pt off 1pt,color=ffqqqq] (0.,-2.)-- (0.,2.);
\draw [dash pattern=on 1pt off 1pt,color=ffqqqq] (0.,2.)-- (-1.72608,1.01027);
\draw [dash pattern=on 1pt off 1pt,color=ffqqqq] (-1.72608,1.01027)-- (-1.14753,0.);
\draw [dash pattern=on 1pt off 1pt,color=ffqqqq] (-1.14753,0.)-- (-1.72608,-1.01027);
\draw [dash pattern=on 1pt off 1pt,color=ffqqqq] (-1.72608,-1.01027)-- (0.,-2.);
\draw [dash pattern=on 1pt off 1pt,color=ffqqqq] (0.,-2.)-- (1.72608,-1.01027);
\draw [dash pattern=on 1pt off 1pt,color=ffqqqq] (1.72608,-1.01027)-- (1.14753,0.);
\draw [dash pattern=on 1pt off 1pt,color=ffqqqq] (1.14753,0.)-- (1.72608,1.01027);
\draw [dash pattern=on 1pt off 1pt,color=ffqqqq] (1.72608,1.01027)-- (0.,2.);
\begin{scriptsize}
\draw [fill=ffqqqq] (-0.950516,1.65769) circle (1.5pt);
\draw [fill=ffqqqq] (-1.7213,0.) circle (1.5pt);
\draw [fill=ffqqqq] (-0.573766,0.) circle (1.5pt);
\draw [fill=ffqqqq] (0.950516,1.65769) circle (1.5pt);
\draw [fill=ffqqqq] (1.7213,0.) circle (1.5pt);
\draw [fill=ffqqqq] (0.573766,0.) circle (1.5pt);
\draw [fill=ffqqqq] (-0.950516,-1.65769) circle (1.5pt);
\draw [fill=ffqqqq] (0.950516,-1.65769) circle (1.5pt);
\draw [fill=ffqqqq] (-1.72608,1.01027) circle (0.5pt);
\draw [fill=ffqqqq] (-1.14753,0.) circle (0.5pt);
\draw [fill=ffqqqq] (1.72608,1.01027) circle (0.5pt);
\draw [fill=ffqqqq] (1.14753,0.) circle (0.5pt);
\draw [fill=ffqqqq] (-1.72608,-1.01027) circle (0.5pt);
\draw [fill=ffqqqq] (1.72608,1.01027) circle (0.5pt);
\draw [fill=ffqqqq] (1.72608,-1.01027) circle (0.5pt);
\draw [fill=qqwuqq] (0.,2.) circle (0.5pt);
\draw [fill=qqwuqq] (0.,-2.) circle (0.5pt);
\draw[] (0.00096496167399,-2.330438200585347174) node { $(ii)$ };
\end{scriptsize}
\end{tikzpicture}
\begin{tikzpicture}[line cap=round,line join=round,>=triangle 45,x=1.0cm,y=1.0cm]
\clip(-2.1431409553721457,-2.454073612323069) rectangle (2.119731706345336,2.1580023224116287);
\draw(0.,0.) circle (2.cm);
\draw (-1.9739043510998817,0.)-- (2.0260956489001183,0.);
\draw [dash pattern=on 1pt off 1pt,color=ffqqqq] (0.,-2.)-- (0.,2.);
\draw [line width=0.4pt,dash pattern=on 1pt off 1pt,color=ffqqqq] (0.,2.)-- (-1.69485,1.06183);
\draw [line width=0.4pt,dash pattern=on 1pt off 1pt,color=ffqqqq] (-1.69485,1.06183)-- (-1.14664,0.);
\draw [line width=0.4pt,dash pattern=on 1pt off 1pt,color=ffqqqq] (-1.14664,0.)-- (-1.87384,-0.699084);
\draw [line width=0.4pt,dash pattern=on 1pt off 1pt,color=ffqqqq] (-1.87384,-0.699084)-- (-0.787249,-1.83854);
\draw [line width=0.4pt,dash pattern=on 1pt off 1pt,color=ffqqqq] (-0.787249,-1.83854)-- (0.7815234120096467,-1.8409837469355346);
\draw [line width=0.4pt,dash pattern=on 1pt off 1pt,color=ffqqqq] (0.7815234120096467,-1.8409837469355346)-- (1.874920313417217,-0.6961851896840988);
\draw [line width=0.4pt,dash pattern=on 1pt off 1pt,color=ffqqqq] (1.874920313417217,-0.6961851896840988)-- (1.14664,0.);
\draw [line width=0.4pt,dash pattern=on 1pt off 1pt,color=ffqqqq] (1.14664,0.)-- (1.69485,1.06183);
\draw [line width=0.4pt,dash pattern=on 1pt off 1pt,color=ffqqqq] (1.69485,1.06183)-- (0.,2.);
\begin{scriptsize}
\draw [fill=ffqqqq] (0.,2.) circle (0.5pt);
\draw [fill=ffqqqq] (0.,0.) circle (0.5pt);
\draw [fill=ffqqqq] (-0.927856,1.67622) circle (1.5pt);
\draw [fill=ffqqqq] (-1.71445,0.0792968) circle (1.5pt);
\draw [fill=ffqqqq] (-0.573319,0.) circle (1.5pt);
\draw [fill=ffqqqq] (-1.40823,-1.34289) circle (1.5pt);
\draw [fill=ffqqqq] (0.,-1.94589) circle (1.5pt);
\draw [fill=ffqqqq] (0.927856,1.67622) circle (1.5pt);
\draw [fill=ffqqqq] (1.71445,0.0792968) circle (1.5pt);
\draw [fill=ffqqqq] (0.573319,0.) circle (1.5pt);
\draw [fill=ffqqqq] (1.40823,-1.34289) circle (1.5pt);
\draw [fill=ffqqqq] (-1.69485,1.06183) circle (0.5pt);
\draw [fill=ffqqqq] (-1.14664,0.) circle (0.5pt);
\draw [fill=ffqqqq] (-1.87384,-0.699084) circle (0.5pt);
\draw [fill=ffqqqq] (-0.787249,-1.83854) circle (0.5pt);
\draw [fill=ffqqqq] (1.69485,1.06183) circle (0.5pt);
\draw [fill=ffqqqq] (1.14664,0.) circle (0.5pt);
\draw [fill=ffqqqq] (1.87384,-0.699084) circle (0.5pt);
\draw [fill=ffqqqq] (0.787249,-1.83854) circle (0.5pt);
\draw[] (0.00096496167399,-2.330438200585347174) node { $(iii)$ };
\end{scriptsize}
\end{tikzpicture}
\begin{tikzpicture}[line cap=round,line join=round,>=triangle 45,x=1.0cm,y=1.0cm]
\clip(-2.1431409553721457,-2.454073612323069) rectangle (2.119731706345336,2.1580023224116287);
\draw(0.,0.) circle (2.cm);
\draw (-1.9739043510998817,0.)-- (2.0260956489001183,0.);
\draw [dash pattern=on 1pt off 1pt,color=ffqqqq] (0.,-2.)-- (0.,2.);
\draw [line width=0.4pt,dash pattern=on 1pt off 1pt,color=ffqqqq] (0.7546649797611029,1.8521557084441291)-- (-0.766932,1.84711);
\draw [line width=0.4pt,dash pattern=on 1pt off 1pt,color=ffqqqq] (-0.766932,1.84711)-- (-1.8497,0.760668);
\draw [line width=0.4pt,dash pattern=on 1pt off 1pt,color=ffqqqq] (-1.8497,0.760668)-- (-1.13908,0.);
\draw [line width=0.4pt,dash pattern=on 1pt off 1pt,color=ffqqqq] (-1.13908,0.)-- (-1.8497,-0.760668);
\draw [line width=0.4pt,dash pattern=on 1pt off 1pt,color=ffqqqq] (-1.8497,-0.760668)-- (-0.766932,-1.84711);
\draw [line width=0.4pt,dash pattern=on 1pt off 1pt,color=ffqqqq] (-0.766932,-1.84711)-- (0.6981350704040327,-1.8741951401793668);
\draw [line width=0.4pt,dash pattern=on 1pt off 1pt,color=ffqqqq] (0.6981350704040327,-1.8741951401793668)-- (1.8497,-0.760668);
\draw [line width=0.4pt,dash pattern=on 1pt off 1pt,color=ffqqqq] (1.8497,0.760668)-- (0.7546649797611029,1.8521557084441291);
\draw [line width=0.4pt,dash pattern=on 1pt off 1pt,color=ffqqqq] (1.8497,0.760668)-- (1.13908,0.);
\draw [line width=0.4pt,dash pattern=on 1pt off 1pt,color=ffqqqq] (1.13908,0.)-- (1.8496973539489154,-0.7606705586483418);
\begin{scriptsize}
\draw [fill=ffqqqq] (0.,0.) circle (0.5pt);
\draw [fill=ffqqqq] (0.,1.94877) circle (1.5pt);
\draw [fill=ffqqqq] (-1.38032,1.37565) circle (1.5pt);
\draw [fill=ffqqqq] (-1.70862,0.) circle (1.5pt);
\draw [fill=ffqqqq] (-0.569538,0.) circle (1.5pt);
\draw [fill=ffqqqq] (0.,-1.94877) circle (1.5pt);
\draw [fill=ffqqqq] (1.38032,1.37565) circle (1.5pt);
\draw [fill=ffqqqq] (1.70862,0.) circle (1.5pt);
\draw [fill=ffqqqq] (0.569538,0.) circle (1.5pt);
\draw [fill=ffqqqq] (-1.38032,-1.37565) circle (1.5pt);
\draw [fill=ffqqqq] (1.38032,-1.37565) circle (1.5pt);
\draw [fill=ffqqqq] (-0.766932,1.84711) circle (0.5pt);
\draw [fill=ffqqqq] (-1.8497,0.760668) circle (0.5pt);
\draw [fill=ffqqqq] (-1.13908,0.) circle (0.5pt);
\draw [fill=ffqqqq] (0.766932,1.84711) circle (0.5pt);
\draw [fill=ffqqqq] (1.8497,0.760668) circle (0.5pt);
\draw [fill=ffqqqq] (1.13908,0.) circle (0.5pt);
\draw [fill=ffqqqq] (-0.766932,-1.84711) circle (0.5pt);
\draw [fill=ffqqqq] (-1.8497,-0.760668) circle (0.5pt);
\draw [fill=ffqqqq] (0.766932,-1.84711) circle (0.5pt);
\draw [fill=ffqqqq] (1.8497,-0.760668) circle (0.5pt);
\draw[] (0.00096496167399,-2.330438200585347174) node { $(iv)$ };
\end{scriptsize}
\end{tikzpicture}
\caption{ } \label{Fig3}
\end{figure}

The following proposition plays an important role in the paper.

\begin{prop} \label{prop000}
Let $\ga_n$ be an optimal set of $n$-means for $P$, and $n\geq 5$. Then, $\ga_n$ contains at least one point from $L_1$, the Voronoi region of which does not contain any point from $L_2$; and at least one point from $L_2$, the Voronoi region of which does not contain any point from $L_1$.
\end{prop}

\begin{proof}
Let $V_n$ denote the $n$th quantization error for any positive integer $n$. By the previous propositions, the lemma is true for $5\leq n\leq 10$. Let $n\geq 11$.  Then, $V_n\leq V_{11}<V_{10}=0.039046$. For the sake of contradiction, assume that for $n\geq 11$, the set $\ga_n$ does not contain any point from $L_1$, the Voronoi region of which does not contain any point from $L_2$. Then,
\[V_n>\int_{L_1}\min_{a\in \set{(-\frac 12 , 0), (0,\frac 12)}}\rho((x, 0), a) dP=\frac 1 4\int_{-1}^0 \rho((t, 0), (-\frac 12, 0)) dt+\frac 14\int_{0}^1 \rho((t, 0), (\frac 12, 0)) dt=\frac{1}{24},\]
implying $V_n>\frac{1}{24}=0.0416667>V_{10}$, which leads to a contradiction.  Hence, $\ga_n$ contains at least one point from $L_1$, the Voronoi region of which does not contain any point from $L_2$. Similarly, we can prove the other part of the proposition. Thus, the proof of the proposition is complete.
\end{proof}

We now state and prove the following theorem, which is the main theorem of this section. Notice that we are saying the theorem as the main theorem of this section, because as mentioned in Remark~\ref{rem22}, this theorem helps us to calculate all the optimal sets of $n$-means, and so, the $n$th quantization errors for all $n\geq 5$ for the mixed distribution $P$.
\begin{theorem} \label{Th0}
Let $n\geq 5$ be a positive integer, and let $\ga_n$ be an optimal set of $n$-means for $P$. Let $3k+2\leq n\leq 3k+4$ for some positive integer $k$. Then, $\ga_n$ contains $k$ elements from $L_1$, the Voronoi regions of which do not contain any point from $L_2$.
\end{theorem}

\begin{proof}
By Proposition~\ref{prop000}, for $n\geq 5$, the set $\ga_n$ always contains points from $L_1$, the Voronoi regions of which do not contain any point from $L_2$, and points from $L_2$, the Voronoi regions of which do not contain any point from $L_1$.
Since the Voronoi region of a point in an optimal set covers maximum area within a shortest distance $P$-almost surely, the set $\ga_n$, given in the theorem, must contain the two points, the Voronoi regions of which contain points from both $L_1$ and $L_2$, in other words, the Voronoi regions of these two points contain points around the two intersections of $L_1$ and $L_2$. Each of the remaining $n-2$ points occurs due to the uniform distribution on $L_1$, or $L_2$, the Voronoi region of which contains points only from $L_1$, or from $L_2$, respectively.

Let $n=n_1+n_2+k+2$ be such that $\ga_n$ contains $k$ elements from $L_1$, the Voronoi regions of which do not contain any point from $L_2$; $n_1$ elements from above the $x_1$-axis, the Voronoi regions of which do not contain any point from $L_1$, and  $n_2$ elements from below the $x_1$-axis, the Voronoi regions of which do not contain any point from $L_1$. Then, there exist three real numbers $a, \, b$, and $c$,  where $-1<a<1$,  $0<b<\frac \pi 2$, and $0<c<\frac \pi 2$, such that the following occur:

$(i)$ The $k$ elements that $\ga_n$ contains from $L_1$ occur due to the uniform distribution on $[-a, a]$, and as mentioned in Theorem~\ref{th31}, are given by the set
\[\set{-a+\frac{2i-1}{k}a : 1\leq i\leq k},\]
with distortion error given by
\begin{align*} &k \Big(\te{distortion error due to the point } -a+\frac a k \te{ in the interval } [-a, -a+\frac{2a} k]\Big)\\
&= \frac k 4 \int_{-a}^{-a+\frac{2a} k}\Big(t-(-a+\frac{a} k)\Big)^2 dt=\frac{a^3}{6 k^2}.
\end{align*}

$(ii)$ The $n_1$ elements that $\ga_n$ contains from above the $x_1$-axis, the Voronoi regions of which do not contain any point from $L_1$, occur due to the uniform distribution on the circular arc $\set{(\cos\gq, \sin \gq) : b\leq \gq\leq \pi-b}$, and by Theorem~\ref{th32}, are given by the set
\begin{equation*} \label{eq0000} \left\{\frac{2n_1}{\pi-2b}\sin\frac{\pi-2b}{2n_1}\Big(\cos(b+(2j-1)\frac{\pi-2b}{2n_1}), \sin(b+(2j-1)\frac{\pi-2b}{2n_1})\Big) : 1\leq j\leq n_1\right\},\end{equation*}
with distortion error
\begin{align*}
&n_1\Big(\frac 1{4\pi} \int_b^{b + \frac{\pi -2 b}{n_1}}\rho\Big((\cos\gq, \sin\gq),  \frac{2 n_1}{\pi -2 b} \sin (\frac{\pi -2 b}{2 n_1} )\Big(\cos(b+\frac{\pi -2 b}{2 n_1}), \sin(b+\frac{\pi -2 b}{2 n_1})\Big)\Big)d\gq\Big)\\
&=\frac{(\pi -2 b)^2-2 n_1^2+2 n_1^2 \cos (\frac{2 b-\pi }{n_1})}{4 \pi  (\pi -2 b)},
\end{align*}
and we denote it by $D_{n_1}$.

$(iii)$ The $n_2$ elements that $\ga_n$ contains from below the $x_1$-axis, the Voronoi regions of which do not contain any point from $L_1$, occur due to the uniform distribution on the circular arc $\set{(\cos\gq, \sin \gq) : \pi+c\leq \gq\leq 2\pi-c}$, and by Theorem~\ref{th32}, are given by the set
\begin{equation*} \label{eq0000} \left\{\frac{2n_2}{\pi-2c}\sin\frac{\pi-2c}{2n_2}\Big(\cos(\pi+c+(2j-1)\frac{\pi-2c}{2n_2}), \sin(\pi+c+(2j-1)\frac{\pi-2c}{2n_2})\Big) : 1\leq j\leq n_2\right\},\end{equation*}
with distortion error
\begin{align*}
&n_2\Big(\frac 1{4\pi} \int_{\pi+c}^{\pi+c + \frac{\pi-2c}{n_2}}\rho\Big((\cos\gq, \sin\gq),  \frac{2 n_2}{\pi-2c} \sin (\frac{\pi-2c}{2 n_2} )\Big(\cos(\pi+c+\frac{\pi-2c}{2 n_2}), \sin(\pi+c+\frac{\pi-2c}{2 n_2})\Big)\Big)d\gq\Big)\\
&=\frac{(\pi -2 c)^2-2 n_2^2+2 n_2^2 \cos (\frac{2 c-\pi }{n_2})}{4 \pi  (\pi -2 c)},
\end{align*}
and we denote it by $D_{n_2}$.

$(iv)$ The two points in $\ga_n$, the Voronoi regions of which contain points from both $L_1$ and $L_2$, are given by the set $\set{(-r, s), (r, s)}$, where
\begin{align*}(-r, s)&=\frac{\frac{1}{4} \int_{-1}^{-a} (t,0)\, dt+\frac 1 {4\pi} \int_{\pi -b}^{ \pi+c } (\cos \theta,\sin\theta )d\theta }{\frac{1}{4} \int_{-1}^{-a}   dt+\frac 1 {4\pi} \int_{\pi -b}^{\pi+c } d\theta}\\
&= \Big(-\frac{-\pi  a^2+2 \sin b+2 \sin c+\pi }{2 (-\pi  a+b+c+\pi )},\frac{\cos c-\cos b}{-\pi  a+b+c+\pi }\Big),
\end{align*}
i.e.,
\[r=\frac{-\pi  a^2+2 \sin b+2 \sin c+\pi }{2 (-\pi  a+b+c+\pi )}, \te{ and } s=\frac{\cos c-\cos b}{-\pi  a+b+c+\pi },\]
and the distortion error for both the two points is given by
\begin{align*} & 2\Big(\frac{1}{4} \int_{-1}^{-a} \rho\Big((t,0), (-r, s)\Big)\Big)\, dt+\frac 1 {4\pi} \int_{\pi -b}^{ \pi+c } \rho\Big((\cos \gq, \sin \gq), (-r, s)\Big)\Big)\,d\theta\Big)\\
&=\frac{1}{24 \pi  (-\pi  a+b+c+\pi )}\Big(\pi ^2 a^4-4 \pi  a^3 b-4 \pi  a^3 c-4 \pi ^2 a^3+12 \pi  \left(a^2-1\right) \sin b\\
&\qquad +12 \pi  a^2 \sin c+6 \pi ^2 a^2-12 \pi  a b-12 \pi  a c-4 \pi ^2 a+12 b^2+24 b c+24 \cos (b+c)+16 \pi  b\\
&\qquad +12 c^2+16 \pi  c-12 \pi  \sin c+\pi ^2-24\Big),
\end{align*}
and we denote it by $D(a, b, c)$.

Let $V(n_1, n_2, k)$ denote the distortion error due to the all above $n_1+n_2+k+2$ elements in $\ga_n$. Then, we have
\begin{align} \label{eq90}
V(n_1, n_2, k)=\frac{a^3}{6 k^2}+D_{n_1}+D_{n_2}+D(a, b, c).
\end{align}
Let $n_1, n_2$, and $k$ be fixed. Then, using the partial derivatives we can obtain the following equations
\begin{equation} \label{eq91} \frac{\pa}{\pa a}(V(n_1, n_2, k))=0, \ \frac{\pa}{\pa b}(V(n_1, n_2, k))=0, \te{ and } \frac{\pa}{\pa c}(V(n_1, n_2, k))=0.
\end{equation}
For a given set of values of $n_1, n_2$, and $k$, solving the equations in \eqref{eq91}, we can obtain the values of $a, b, c$. Putting the values of $a, b, c$ in \eqref{eq90}, we can obtain the distortion error for the given set of values of $n_1, n_2, k$.

Now, to prove the theorem we use induction on $k$. If $k=1$, and $k=2$, the theorem is true due to the previous propositions. Let us assume that the theorem is true for $k=m$, i.e., when $3m+2\leq n\leq 3m+4$. We now prove that the theorem is true for $3(m+1)+2\leq n\leq 3(m+1)+4$.
By the assumption, the theorem is true for $n=3m+4$, i.e., the set $\ga_{3m+4}$ contains $m$ points from $L_1$, the Voronoi regions of which do not contain any point from $L_2$, and $(2m+2)$ points occur due to the uniform distribution on $L_2$, the Voronoi region of which do not contain any point from $L_1$. Again, due to the mixed distribution with equal weights to the component probabilities, and symmetry of the circle with respect to the $x_1$-axis, we can can assume that $\ga_n$ contains $m+1$ elements from above, and $m+1$ elements from below. Now, to calculate $\ga_{n+1}$, we need to add one extra point either to $L_1$, or $L_2$ in an optimal way, i.e., the Voronoi regions of the new point will contain only the points from $L_1$, or from $L_2$, and the overall distortion error due to $n+1$ points becomes smallest.
First suppose that the extra point is added to $L_1$, the Voronoi region of which does not contain any point from $L_2$. As described above using \eqref{eq90}, we calculate the distortion error $V(m+1, m+1, m+1)$. Next, suppose that the extra point is added to $L_2$, the Voronoi region of which does not contain any point from $L_1$, and using \eqref{eq90}, we calculate the distortion error $V(m+2, m+1, m)$, or $V(m+1, m+2, m)$. We see that the distortion error $V(m+1, m+1, m+1)$ is the smallest, which implies the fact that $\ga_{n+1}$ contains $m+1$ points from $L_1$. Once, $\ga_{n+1}$ is known, similarly we can obtain $\ga_{n+2}$, and $\ga_{n+3}$ with distortion errors, respectively, $V(m+1, m, m+1)$ and $V(m+1, m+1, m+1)$. Thus, we see that each of $\ga_{n+1}$, $\ga_{n+2}$, and $\ga_{n+3}$ contains $m+1$ points from $L_1$, the Voronoi regions of which do not contain any point from $L_2$. Notice that $n+1=3(m+1)+2$, $n+2=3(m+1)+3$, and $n+1=3(m+1)+4$, i.e., for the positive integer $n$ satisfying $3(m+1)+2\leq n\leq 3(m+1)+4$, the set $\ga_n$ contains $m+1$ elements from $L_1$, the Voronoi regions of which do not contain any point from $L_2$. Thus, the theorem is true for $k=m+1$ if it is true for $k=m$. Hence, by the principle of mathematical induction, the theorem is true for all positive integers $k$, and thus, the proof of the theorem is complete.
\end{proof}

\begin{remark} \label{rem22}
For $n\geq 5$, let $3k+2\leq n\leq 3k+4$ for some positive integer $k$. Then, by Theorem~\ref{Th0}, we can say that if $n-k-2$ is an even number, then an optimal set of $n$-means contains $\frac 12(n-k-2)$ elements from either side of the $x_1$-axis, the Voronoi regions of which do not contain any point from $L_1$; and if $n-k-2$ is an odd number, then an optimal set of $n$-means contains $\frac 12\lfloor n-k-2\rfloor$ elements from one side of the $x_1$-axis, and $\frac 12\lfloor n-k-2\rfloor+1$ elements from the other side of the $x_1$-axis, the Voronoi regions of which do not contain any point from $L_1$. Thus, by Theorem~\ref{Th0}, using Theorem~\ref{th31}, and Theorem~\ref{th32}, we can easily determine the optimal sets of $n$-means and the $n$th quantization errors for all $n\geq 5$.
\end{remark}

The following proposition gives the quantization dimension and the quantization coefficient for the mixed distribution.
\begin{prop} \label{prop90}
Quantization dimension $D(P)$ of the mixed distribution $P$ is one, which is the dimension of the underlying space, and the quantization coefficient exists as a finite positive number which equals $\frac{3}{8} \left(4+\pi ^2\right)$.
\end{prop}
\begin{proof} By Remark~\ref{rem22}, we see that if $n$ is of the form $n=3k+2$ for some positive integer $k$, then $\ga_n$ contains $k$ elements from $L_1$, the Voronoi regions of which do not contain any point from $L_2$, and $k$ elements from the above, and $k$ elements from below the $x_1$-axis, the Voronoi region of which do not contain any point from $L_1$.
For $n\in \D N$, $n\geq 5$, let $\ell(n)$ be the unique positive integer such that $3\ell(n)+2 \leq n<3(\ell(n)+1)+2$. Then,
$V_{3(\ell(n)+1)+2}<V_n\leq V_{3\ell(n)+2}$ implying
\begin{align} \label{eq45}
\frac {2 \log (3\ell(n)+2)}{-\log V_{3(\ell(n)+1)+2}}<\frac {2\log n}{-\log V_n} <\frac{2 \log (3(\ell(n)+1)+2)}{-\log V_{3\ell(n)+2}}.
\end{align}
Notice that if $n \to \infty$, then $\ell(n) \to \infty$. Moreover, if $n\to \infty$, they by \eqref{eq90} and \eqref{eq91}, we can see that $a\to 1$, $b\to 0$, and $c\to 0$. Assume that $n$ is sufficiently large, in other words, assume that $\ell(n)$ is sufficiently large, and then as $a\to 1$, $b\to 0$, and $c\to 0$, by \eqref{eq90} we have $D(a, b, c)\to 0$, implying
\[V_{3\ell(n)+2}=V(\ell(n), \ell(n), \ell(n))=\frac{-6 \ell(n)^4+6 \ell(n)^4 \cos \frac{\pi }{\ell(n)}+3 \pi ^2 \ell(n)^2+\pi ^2}{6 \pi ^2 \ell(n)^2},\]
yielding
\[\lim_{n\to \infty} \frac {2 \log (3\ell(n)+2)}{-\log V_{3(\ell(n)+1)+3}}=\underset{\ell(n)\to \infty }{\text{lim}}\frac{2 \log (3 \ell(n)+2)}{-\log \Big(\frac{-6 (\ell(n)+1)^4+3 \pi ^2 (\ell(n)+1)^2+6 (\ell(n)+1)^4 \cos (\frac{\pi }{\ell(n)+1})+\pi ^2}{6 \pi ^2 (\ell(n)+1)^2}\Big)}=1,\]
and \[\lim_{n\to \infty} \frac{2 \log (3(\ell(n)+1)+2)}{-\log V_{3\ell(n)+2}}=\underset{\ell(n)\to \infty }{\text{lim}}\frac{2 \log (3 (\ell(n)+1)+2)}{-\log \Big(\frac{-6 \ell(n)^4+6 \ell(n)^4 \cos (\frac{\pi }{\ell(n)})+3 \pi ^2 \ell(n)^2+\pi ^2}{6 \pi ^2 \ell(n)^2}\Big)}=1\]
and hence, by \eqref{eq45}, $\mathop{\lim}\limits_{n\to \infty} \frac {2\log n}{-\log V_n}=1$, which is the dimension of the underlying space. Again,
\begin{equation} \label{eq46} (3\ell(n)+2)^2V_{3(\ell(n)+1)+2}<n^2 V_n<(3(\ell(n)+1)+2)^2V_{3\ell(n)+2}.
\end{equation}
We have
\begin{align*}
&\lim_{n \to \infty} (3\ell(n)+2)^2V_{3(\ell(n)+1)+2}\\
&=\underset{\ell(n)\to \infty }{\text{lim}}(3 \ell(n)+2)^2 \frac{-6 (\ell(n)+1)^4+3 \pi ^2 (\ell(n)+1)^2+6 (\ell(n)+1)^4 \cos (\frac{\pi }{\ell(n)+1})+\pi ^2}{6 \pi ^2 (\ell(n)+1)^2}\\
&=\frac{3}{8} \left(4+\pi ^2\right),
\end{align*}
and
\begin{align*}
&\lim_{n \to \infty} (3(\ell(n)+1)+2)^2V_{3\ell(n)+2}\\
&=\underset{\ell(n)\to \infty }{\text{lim}}(3 (\ell(n)+1)+2)^2 \frac{-6 \ell(n)^4+6 \ell(n)^4 \cos (\frac{\pi }{\ell(n)})+3 \pi ^2 \ell(n)^2+\pi ^2}{6 \pi ^2 \ell(n)^2}=\frac{3}{8} \left(4+\pi ^2\right),
\end{align*}
and hence, by \eqref{eq46} we have
$\mathop{\lim}\limits_{n\to\infty} n^2 V_n=\frac{3}{8} \left(4+\pi ^2\right)$, i.e., the quantization coefficient exists as a finite positive number which equals $=\frac{3}{8} \left(4+\pi ^2\right)$.
Thus, the proof of the proposition is complete.
\end{proof}

\section{Optimal quantization for the mixture of two uniform distributions on two disconnected line segments} \label{sec3}

Let $P_1$ and $P_2$ be two uniform distributions, respectively, on the intervals $[0, \frac 12]$ and $[\frac 34, 1]$. Write
 \[J_1:=[0, \frac 12], \te{ and } J_2:=[\frac 34, 1].\]
 Let $f_1$ and $f_2$ be their respective density functions. Then, $f_1(x)=2$ if $x\in [0,\frac 12]$, and zero, otherwise; and $f_2(x)=4$ if $x\in [\frac 34, 1]$, and zero, otherwise. Let $P:=\frac 34 P_1+\frac 14 P_2$. In the sequel, for the mixed distribution $P$, we determine the optimal sets of $n$-means and the $n$th quantization errors for all positive integers $n$. By $E(P)$ and $V(P)$, we mean the expectation and the variance of a random variable with distribution $P$. By $\ga_n(\mu)$, we denote an optimal set of $n$-means with respect to a probability distribution $\mu$, and $V_n(\mu)$ represents the corresponding quantization error for $n$-means. If $\mu$ is the mixed distribution $P$, in the sequel, we sometimes denote it by $\ga_n$ instead of $\ga_n(P)$, and the corresponding quantization error is denoted by $V_n$ instead of $V_n(P)$.

\begin{lemma}
Let $P$ be the mixed distribution defined by $P=\frac 34 P_1+\frac 1 4 P_2$. Then, $E(P)=\frac{13}{32}$, and $V(P)=\frac{277}{3072}$.
\end{lemma}

\begin{proof}
We have
\[E(P)=\int x dP=\frac 34 \int x d(P_1(x))+\frac 14 \int x d(P_2(x))=\frac 34 \int_0^{\frac{1}{2}} 2x \, dx+\frac 14 \int_{\frac{3}{4}}^1 4 x \, dx\]
yielding $E(P)=\frac{13}{32}$, and
\[V(P)=\int (x-E(P))^2 dP=\frac 34 \int(x-E(P))^2 d(P_1(x))+\frac 14 \int(x-E(P))^2 d(P_2(x)),\]
implying
$V(P)=\frac{277}{3072}$, and thus, the lemma is yielded.
\end{proof}

\begin{remark}
The optimal set of one-mean is the set $\set{\frac{13}{32}}$, and the corresponding quantization error is the variance $V:=V(P)$ of a random variable with distribution $P$.
\end{remark}

\begin{lemma} \label{lemma2}
The set $\ga:=\set{ \frac 14, \frac 7 8}$ is an optimal set of two-means, and the corresponding quantization error is given by $V_2=\frac{13}{768}$.
\end{lemma}
\begin{proof} Consider the set of two points $\gb$ given by $\gb:=\set{ \frac 14, \frac 7 8}$. The distortion error due to the set $\gb$ is given by
\begin{align*} &\int\min_{a\in \gb} (x-a)^2 dP=\int_{J_1}(x-\frac 14)^2 dP+\int_{J_2} (x-\frac 78)^2 dP=\frac{3}{4} \int_0^{\frac{1}{2}} 2 (x-\frac{1}{4})^2 \, dx+\frac{1}{4} \int_{\frac{3}{4}}^1 4 (x-\frac{7}{8})^2 \, dx\\
&=\frac{13}{768}=0.0169271.
\end{align*}
Since $V_2$ is the quantization error for two-means, we have $V_2\leq 0.0169271$. Let $\ga:=\set{a_1, a_2}$ be an optimal set of two-means. Since the points in an optimal set are the conditional expectations in their own Voronoi regions, without any loss of generality, we can assume that $0<a_1<a_2<1$. We now show that the Voronoi region of $a_1$ does not contain any point from $J_2$, and the Voronoi region of $a_2$ does not contain any point from $J_1$. Suppose that  $\frac{13}{40}\leq a_1$. Then,
\[V_2>\int_{[0, \frac{13}{40}]} (x-\frac{13}{40})^2 dP=\frac{2197}{128000}=0.0171641>V_2,\]
which is a contradiction, and so, we can assume that $a_1<\frac{13}{40}<\frac 12$. Since $a_1<\frac {13}{40}$, the Voronoi region of $a_1$ does not contain any points from $J_2$. If it contains points from $J_2$, then
$\frac 12(a_1+a_2)>\frac 34$, implying $a_2>\frac 32-a_1\geq \frac 32-\frac {13}{40} =\frac {47}{40}>1$, which is a contradiction. Hence, we can assume that \begin{equation}\label{eq001}
a_1\leq E(X : X \in J_1)=\frac 14, \te{ and } a_2\leq E(X : X\in J_2)=\frac 78.
\end{equation}
Suppose that $a_2<\frac 58$. Then,
\[V_2> \frac{1}{4} \int_{\frac{3}{4}}^1 4 (x-\frac{5}{8})^2 \, dx=\frac{13}{768}=0.0169271\geq V_2,\]
which leads to a contradiction. So, we can assume that $\frac 58\leq a_2$. Thus, by \eqref{eq001}, we have $\frac 58\leq a_2\leq \frac 78$.  Assume that $\frac 58\leq a_2\leq \frac 34$. Since $a_1\leq \frac 14$, the following cases can arise:

\tbf{Case~1.} $\frac 18\leq a_1\leq \frac 14$.

Then, notice that $\frac {13}{32}<\frac 12(\frac 14+\frac 58)=\frac 7{16}<\frac 12$, and so,
 \begin{align*}
&\int_{[0, \frac {13}{32}]}\min_{a \in \set{a_1, a_2}}(x-a)^2 dP=\frac{13 \left(3072 a_1^2-1248 a_1+169\right)}{65536},
 \end{align*}
the minimum value of which is $\frac{2197}{262144}$, and it occurs when $a_1=\frac{13}{64}$. Notice that for $a_1=\frac{13}{64}$, we have $\frac{13}{32}=0.40625<\frac{1}{2} (\frac{13}{64}+\frac{5}{8})=0.414063$.
Thus, we have
\begin{align*}
V_2\geq \frac{2197}{262144} +\frac{3}{4} \int_{\frac{13}{32}}^{\frac{7}{16}} 2  (x-\frac{1}{4} )^2 \, dx+\frac{3}{4} \int_{\frac{7}{16}}^{\frac{1}{2}} 2 (x-\frac{5}{8} )^2 \, dx+\frac{1}{4} \int_{\frac{3}{4}}^1 4  (x-\frac{3}{4})^2 \, dx=\frac{13603}{786432},
\end{align*}
yielding $V_2\geq 0.0172971>V_2$, which is a contradiction.

\tbf{Case~2.} $a_1<\frac 18$.

Then, $\frac 12(\frac 18+\frac 58)=\frac 38<\frac 12$, and so
\[V_3\geq \frac{3}{4} \int_{\frac{1}{8}}^{\frac{3}{8}} 2  (x-\frac{1}{8} )^2 \, dx+\frac{3}{4} \int_{\frac{3}{8}}^{\frac{1}{2}} 2 (x-\frac{5}{8} )^2 \, dx+\frac{1}{4} \int_{\frac{3}{4}}^1 4  (x-\frac{3}{4} )^2 \, dx=\frac{61}{3072}=0.0198568>V_3,\]
which leads to a contradiction.

Hence, by Case~1 and Case~2, we can conclude that $\frac 34\leq a_2\leq \frac 78$. Suppose that $\frac 34\leq a_2\leq \frac {13}{16}$. Then, the Voronoi region of $a_2$ must contain points prom $J_1$ implying $\frac 12(a_1+a_2)<\frac 12$, which yields $a_1<1-a_2\leq 1-\frac 34=\frac 14$. Again, \[\int_{J_1} (x-a_1)^2 dP=\frac{1}{16} (12 a^2-6 a+1),\]
the minimum value of which is $\frac{1}{64}$ when $a_1=\frac 14$.
Thus, we have
\[V_2\geq \int_{J_1}(x-\frac 14)^2 dP+\int_{J_2}(x-\frac {13}{16})^2 dP=\frac{55}{3072}=0.0179036>V_2,\]
which gives a contradiction.
Hence, we can assume that $\frac{13}{16}<a_2\leq \frac 78$. Suppose that the Voronoi region of $a_2$ contains points from $J_1$, i.e., $\frac 12(a_1+a_2)<\frac 12$. Then, $a_1<1-a_2\leq 1-\frac{13}{16}=\frac {3}{16}$. Notice that
\[\int_{J_1} (x-a_1)^2 dP=\frac{1}{16} (12 a_1^2-6 a_1+1),\]
the minimum value of which is $\frac{19}{1024}$ when $a_1=\frac 3{16}$.
Thus, we have $V_2\geq \frac{19}{1024}=0.0185547>V_2$, which is a contradiction. Thus, we can assume that the Voronoi region of $a_2$ does not contain any point from $J_1$. Previously, we have proved that the Voronoi region of $a_1$ does not contain any point from $J_2$. Hence, we have
$a_1=E(X : X \in J_1)=\frac 14, \te{ and } a_2= E(X : X\in J_2)=\frac 78$, and the corresponding quantization error for two-means is given by $V_2=\frac{13}{768}$.
\end{proof}

\begin{lemma} \label{lemma3}
The set $\set{\frac 18, \frac 38, \frac 78}$ forms an optimal set of three-means with quantization error $V_3=\frac{1}{192}$.
\end{lemma}
\begin{proof}
Consider the set of three points $\gb$, such that $\gb:=\set{\frac 18, \frac 38, \frac 78}$. The distortion error due to the set $\gb$ is given by
\begin{align*} &\int\min_{a\in \gb} (x-a)^2 dP=2\cdot \frac{3}{4}  \int_0^{\frac{1}{4}} 2 (x-\frac{1}{8})^2 \, dx+\frac{1}{4} \int_{\frac{3}{4}}^1 4 (x-\frac{7}{8})^2 \, dx=\frac{1}{192}.
\end{align*}
Since $V_3$ is the quantization error for three-means, we have $V_3\leq \frac{1}{192}= 0.00520833$. Let $\ga:=\set{a_1, a_2, a_3}$ be an optimal set of three-means. Since the points in an optimal set are the conditional expectations in their own Voronoi regions, without any loss of generality, we can assume that $0<a_1<a_2<a_3<1$. We now show that $a_2<\frac 12$, and $\frac 34<a_3$. If $a_3<\frac 34$, then
\[V_3>\int_{J_2}(x-\frac 34)^2 dP=\frac{1}{4} \int_{\frac{3}{4}}^1 4 (x-\frac{3}{4})^2 \, dx= \frac{1}{192}= 0.00520833\geq V_3,\]
which leads to a contradiction.
Hence, we can assume that $\frac 34<a_3$. Next, we show that $a_2<\frac 12$. Suppose that $\frac 12\leq a_2$. Then,
\begin{align*}
&\int_{J_1}\min_{a \in \set{a_1, \frac 12}}(x-a)^2 dP=\frac{3}{4} \int_0^{\frac{1}{2} \left(a_1+\frac{1}{2}\right)} 2 (x-a_1)^2 \, dx+\frac{3}{4} \int_{\frac{1}{2} \left(a_1+\frac{1}{2}\right)}^{\frac{1}{2}} 2 \Big(x-\frac{1}{2}\Big)^2 \, dx\\
&=\frac{1}{64} (24 a_1^3+12 a_1^2-6 a_1+1),
\end{align*}
the minimum value of which is $\frac{1}{144}$, and it occurs when $a_1=\frac 16$. Thus, in this case, we see that $V_3\geq \frac{1}{144}=0.00694444>V_3$, which leads to a contradiction. Hence, we can assume that $0<a_1<a_2<\frac 12$. Suppose that the Voronoi region of $a_2$ contains points from $J_2$. Then, $\frac 12(a_2+a_3)>\frac 34$ implying $a_3>\frac 32-a_1\geq \frac 32-\frac 12=1$, which is a contradiction, as $a_3<1$. Thus, we see that the Voronoi region of $a_2$ does not contain any point from $J_2$. Suppose that the Voronoi region of $a_3$ contains points from $J_1$. Then, $\frac 12(a_2+a_3)<\frac 12$ implying $a_2<1-a_3\leq 1-\frac 34=\frac 14$, and so
\[V_3>\frac{3}{4} \int_{\frac{1}{4}}^{\frac{1}{2}} 2 (x-\frac{1}{4})^2 \, dx=\frac{1}{128}=0.0078125>V_3,\]
which is a contradiction. So, we can assume that the Voronoi region of $a_3$ does not contain any point from $J_1$. Thus, by Theorem~\ref{th31}, we can conclude that $a_1=\frac 18$, $a_2=\frac 38$, and $a_3=\frac 78$, and
\[V_3=\int\min_{a\in \ga} (x-a)^2 dP=\frac{1}{192},\]
which completes the proof of the lemma.
\end{proof}

\begin{remark} \label{rem1}
By Lemma~\ref{lemma2}, and Lemma~\ref{lemma3}, we see that $\ga_2=\ga_1(P_1)\uu\ga_1(P_2)$, and $\ga_3=\ga_2(P_1)\uu \ga_1(P_2)$. Using the similar technique, we can show that $\ga_4=\ga_3(P_1)\uu \ga_1(P_2)$, $\ga_5=\ga_3(P_1)\uu \ga_2(P_2)$, $\ga_6=\ga_4(P_1)\uu \ga_2(P_2)$, $\ga_7=\ga_5(P_1)\uu \ga_2(P_2)$,   $\ga_8=\ga_6(P_1)\uu \ga_2(P_2)$, and $\ga_9=\ga_6(P_1)\uu \ga_3(P_2)$.
\end{remark}

We now prove the following propositions.

\begin{prop}\label{prop0001}
 Let $\ga_n$ be an optimal set of $n$-means for $P$ for $n\geq 2$. Then, the set $\ga_n$ does not contain any point from the open interval $(\frac 12, \frac 34)$.
\end{prop}
\begin{proof}
By Remark~\ref{rem1}, the proposition is true for $2\leq n\leq 9$. We now prove that the proposition is true for any positive integer $n\geq 10$. Take any $n\geq 10$.  Since $\ga_9=\ga_6(P_1)\uu \ga_3(P_2)$, and the Voronoi region of any point in $\ga_9\ii J_1$ does not contain any point from $J_2$, and the Voronoi region of any point in $\ga_9\ii J_2$ does not contain any point from $J_1$, we have
\[V_9=\frac 34 V_6(P_1)+\frac 14 V_3(P_2)=\frac{1}{1728}=0.000578704.\]
Since $V_n$ is the quantization error for $n$-means for $n\geq 10$, we have $V_n\leq V_9= 0.000578704$. Let $\ga_n:=\set{a_1, a_2, \cdots, a_n}$ be an optimal set of $n$-means for $P$ such that $a_1<a_2<\cdots<a_n$.
 Let $j=\max\set{i : a_i\leq \frac 12}$. Then, $a_j\leq \frac 12<a_{j+1}$. The proposition will be proved if we can show that $a_{j+1}\in J_2$. For the sake of contradiction, assume that $a_{j+1} \in (\frac 12,\frac 34)$. Then, the following two cases can arise:

\tbf{Case~1.} $\frac 12<a_{j+1}\leq \frac 58$.

In this case, the Voronoi region of $a_{j+1}$ must contain points from $J_2$, otherwise, the quantization error can be strictly reduced my moving the point $a_{j+1}$ to $\frac 12$. Thus, $\frac 12(a_{j+1}+a_{j+2})>\frac 34$ implying $a_{j+2}>\frac 32-a_{j+1}\geq \frac 32-\frac 58=\frac 78$, which yields the fact that \[V_n\geq \int_{[\frac 34, \frac 78]}(x-\frac 78)^2 dP=\frac{1}{4} \int_{\frac{3}{4}}^{\frac{7}{8}} 4(x-\frac{7}{8})^2 \, dx=0.000651042>V_n,\]
which leads to a contradiction.

\tbf{Case~2.} $\frac 58\leq a_{j+1}< \frac 34$.

In this case, we have $\frac 12(a_{j}+a_{j+1})<\frac 12$ implying $a_{j}<1-a_{j+1}\leq 1-\frac 58=\frac 38$, which yields the fact that
\[V_n\geq \int_{[\frac 38, \frac 12]}(x-\frac 38)^2 dP=\frac{3}{4} \int_{\frac{3}{8}}^{\frac{1}{2}} 2(x-\frac{3}{8})^2 \, dx=0.000976563>V_n,\]
which is a contradiction.

In light of the above two cases, we can conclude that $a_{j+1}\nin (\frac 12, \frac 34)$. Hence, $\frac 34<a_{j+2}$, i.e., $a_{j+2}\in J_2$. Thus, the proof of the proposition is complete.
\end{proof}

\begin{prop} \label{prop02}
Let $\ga_n$ be an optimal set of $n$-means for $P$ for $n\geq 2$. Then, for $n\geq 2$, $\ga_n\ii J_1\neq \es$, and $\ga_n\ii J_2\neq \es$. Moreover, for $n\geq 2$, any point in $\ga_n\ii J_1$ does not contain any point from $J_2$, and any point in $\ga_n\ii J_2$ does not contain any point from $J_1$,
\end{prop}

\begin{proof}
As shown in the proof of Lemma~\ref{lemma2}, and Lemma~\ref{lemma3}, we see that the proposition is true for $n=2, 3$. By Lemma~\ref{lemma3}, we know $V_3=\frac{1}{192}= 0.00520833$. We now prove the proposition for $n\geq 4$. Let $n\geq 4$.
Since $V_n$ is the quantization error for $n$-means for $n\geq 4$, we have $V_n\leq V_3= 0.00520833$.
Let $\ga_n:=\set{a_1, a_2, \cdots, a_n}$ be an optimal set of $n$-means for $P$ such that $a_1<a_2<\cdots<a_n$. If $\ga_n\ii J_2=\es$, then
\[V_n>\frac{1}{4} \int_{\frac{3}{4}}^1 4 (x-\frac{3}{4})^2 \, dx=0.00520833,\]
which is a contradiction as $V_n\leq 0.00520833$. On the other hand, if $\ga_n\ii J_1=\es$, then
\[V_n>\frac{3}{4} \int_0^{\frac{1}{2}} 2 (x-\frac{1}{4})^2 \, dx=\frac{1}{64}=0.015625>V_n,\]
which leads to a contradiction. Hence,  $\ga_n\ii J_1\neq \es$, and $\ga_n\ii J_2\neq \es$.
 Let $j=\max\set{i : a_i\leq \frac 12}$. Then, $a_j\leq \frac 12$, and due to Proposition~\ref{prop0001}, we have $\frac 34\leq a_{j+1}$. If the Voronoi region of $a_j$ contains points from $J_2$, then $\frac 12(a_j+a_{j+1})>\frac 34$ implying $a_{j+1}>\frac 32-a_j\geq \frac 32-\frac 12=1$, which is a contradiction.  If the Voronoi region of $a_{j+1}$ contains points from $J_1$, then $\frac 12(a_j+a_{j+1})<\frac 12$ implying $a_{j}<1-a_{j+1}\leq 1-\frac 34=\frac 14$. Then,
 \[V_n\geq \int_{[\frac 14, \frac 12]}(x-\frac 14)^2 dP=\frac{3}{4} \int_{\frac{1}{4}}^{\frac{1}{2}} 2 (x-\frac{1}{4})^2 \, dx=\frac{1}{128}\]
 yielding  $V_n\geq 0.0078125>V_n$, which leads to a contradiction. Thus, the proof of the proposition is complete.
\end{proof}

\begin{defi} \label{defi59}
For $n\in\D N$, and $n\geq 2$, define the function $a(n)$ as follows:
\[a(n)=\min\set {k\in \D N : H(n, k)>0}, \]
where $H(n, k)= \frac 1{n^3} -\sum_{i=k}^\infty \frac 1{(i+1)^4}$.
\end{defi}

\begin{remark}
Notice that $\sum_{i=k}^\infty \frac 1{(i+1)^4}$ is a decreasing function of $k\in \D N$, and so for a given $n\geq 2$, $H(n, k)$ is an increasing function of $k$, and thus the function $a(n)$ is well defined. Moreover, $\set{\frac 1{n^3}}_{n\geq 2}$ is a decreasing sequence, and so, the sequence $\{a(n)\}_{n=2}^\infty$ is an increasing sequence. In fact,
\[\set{a(n)}_{n=2}^\infty=\set{1, 2, 3, 3, 4, 5, 6, 6, 7, 8, 8, 9, 10, 10, 11, 12, 12, 13, 14, 15, 15, 16, 17, 17, 18, 19, \cdots}.\]
By $\lfloor x\rfloor$ it is meant the greatest integer not exceeding $x$.
To find the value of $a(n)$ for any positive integer $n$, one can start checking by putting $k=\lfloor \frac {2n}3\rfloor$ in the function $H(n, k)$. If $H(n, k)>0$ then find $H(n, k-1), H(n, k-2), \cdots$ until one obtains some positive integer $m$, such that $H(n, m)>0$, and $H(n, m-1)<0$, and then $a(n)=m$. If $H(n, k)<0$ then find $H(n, k+1), H(n, k+2), \cdots$ until one obtains some positive integer $m$, such that $H(n, m)>0$, and $H(n, m-1)<0$, and then $a(n)=m$.
\end{remark}

\begin{remark} \label{rem34}
For $n\geq 2$ let $\ga_n$ be an optimal set of $n$-means for $P$. Due to Proposition~\ref{prop0001} and Proposition~\ref{prop02}, we can conclude that if $\ga_n$ contains $k$ elements from $J_1$, then $\ga_n$ contains $n-k$ elements from $J_2$. Thus, we have
\[V_n:=V_n(P)=\int\min_{a\in \ga_n}(x-a)^2 dP=\frac 34 \int\min_{a\in \ga_n\ii J_1}(x-a)^2 dP_1+\frac 14 \int\min_{a\in \ga_n\ii J_2}(x-a)^2 dP_2,\]
yielding
\[V_n(P)=\frac 34 V_{k}(P_1)+\frac 14 V_{n-k}(P_2).\]
\end{remark}

Let us now give the following theorem, which gives the optimal sets of $n$-means and the $n$th quantization errors for the mixed distribution $P$ for all positive integers $n\geq 2$.

\begin{theorem} \label{main1}
For $n\geq 2$, let $\ga_n$ be an optimal set of $n$-means for $P$. Then, $\ga_n$ contains $a(n)$ elements from $J_1$, i.e.,
\[\ga_n(P)=\ga_{a(n)}(P_1)\uu \ga_{n-a(n)}(P_2), \te{ and } V_n(P)=\frac 34 V_{a(n)}(P_1)+\frac 1 4 V_{n-a(n)}(P_2).\]
\end{theorem}

\begin{proof}
Assume that $\ga_n$ contains $k$ elements from $J_1$. Let $V(k, n-k)$ is the corresponding distortion error. Then, as mentioned in Remark~\ref{rem34}, we have
\[V(k, n-k)=\frac 34 V_{k}(P_1)+\frac 1 4 V_{n-k}(P_2).\]
Notice that if our assumption is correct, then we must have $V_n=V(k, n-k)$.

Let us now run the following algorithm:

$(i)$ Write $k:=\lfloor \frac {2n}3\rfloor$.

$(ii)$ If $V(k-1, n-(k-1))<V(k, n-k)$ replace $k$ by $k-1$ and return, else go to step $(iii)$.

$(iii)$ If $V(k+1, n-(k+1))<V(k, n-k)$ replace $k$ by $k+1$ and return, else step $(iv)$.

$(iv)$ End.

After running the above algorithm, we see that $k=a(n)$, i.e., our assumption is correct. Thus, the proof of the theorem is complete.
\end{proof}

\begin{remark}
If $n=14$, then $k=\lfloor \frac {28} 3\rfloor=9$. By running the algorithm as mentioned in the theorem, we obtain $k=10$. Moreover, notice that $a(14)=10$, i.e., $\ga_{14}$ contains $a(14)$ elements from $J_1$, and $n-a(14)$ elements from $J_2$, i.e., $\ga_{14}=\ga_{a(14)}(P_1)\uu \ga_{14-a(14)}(P_2)$. If $n=100$, then $k=\lfloor \frac {200} 3\rfloor=66$. By running the algorithm as mentioned in the theorem, we obtain $k=69$. Moreover, we have $a(100)=69$, i.e., $\ga_{100}$ contains $a(100)$ elements from $J_1$, and $n-a(100)$ elements from $J_2$, i.e., $\ga_{100}=\ga_{a(100)}(P_1)\uu \ga_{100-a(100)}(P_2)$.
\end{remark}

\section{Optimal quantization for the mixture of two uniform distributions on two connected line segments} \label{sec4}
 Let $P_1$ and $P_2$ be two uniform distributions, respectively, on the intervals $[0, \frac 12]$ and $[\frac 12, 1]$. Write
 \[J_1:=[0, \frac 12], \te{ and } J_2:=[\frac 12, 1].\]
  Let $f_1$ and $f_2$ be their respective density functions. Then, $f_1(x)=2$ if $x\in [0,\frac 12]$, and zero, otherwise; and $f_2(x)=2$ if $x\in [\frac 12, 1]$, and zero, otherwise. Let $P:=\frac 34 P_1+\frac 14 P_2$.
For such a mixed distribution, in this section, we investigate the optimal sets of $n$-means and the $n$th quantization errors for all $n\in \D N$. Notice that the density function of the mixed distribution $P$ can be written as follows:
\[f(x)=\left\{\begin{array}{cc}
\frac 32  & \te{ if } x\in J_1,\\
\frac 12  & \te{ if } x\in J_2,\\
0 & \te{ otherwise}.
\end{array}\right.\]

Let us now prove the following lemma.

\begin{lemma}
Let $P$ be the mixed distribution defined by $P=\frac 34 P_1+\frac 1 4 P_2$. Then, $E(P)=\frac{3}{8}$, and $V(P)=\frac{13}{192}$.
\end{lemma}

\begin{proof}
We have
\[E(P)=\int x dP=\frac 34 \int x d(P_1(x))+\frac 14 \int x d(P_2(x))=\frac 34 \int_0^{\frac{1}{2}} 2x \, dx+\frac 14 \int_{\frac{1}{2}}^1 2 x \, dx\]
yielding $E(P)=\frac{3}{8}$, and
\[V(P)=\int (x-E(P))^2 dP=\frac 34 \int(x-E(P))^2 d(P_1(x))+\frac 14 \int(x-E(P))^2 d(P_2(x)),\]
implying
$V(P)=\frac{13}{192}$, and thus, the lemma is yielded.
\end{proof}
\begin{remark}
The optimal set of one-mean is the set $\set{\frac{3}{8}}$, and the corresponding quantization error is the variance $V:=V(P)$ of a random variable with distribution $P$.
\end{remark}

\begin{prop} \label{prop01}
For $n\geq 2$, let $\ga_n$ be an optimal set of $n$-means. Then, $\ga_n\ii J_1\neq \es$, and $\ga_n\ii J_2\neq \es$.
\end{prop}

\begin{proof}
Consider the set of two points $\gb:=\set{\frac 14, \frac 34}$. The distortion error due to the set $\gb$ is given by
\begin{align*}
&\int\min_{b\in \gb} (x-b)^2 dP=\int_{J_1}(x-\frac 14)^2 dP+\int_{J_2}(x-\frac 34)^2 dP\\
&=\frac{3}{4} \int_0^{\frac{1}{2}} 2 \Big(x-\frac{1}{4}\Big)^2 \, dx+\frac{1}{4} \int_{\frac{1}{2}}^1 2 \Big(x-\frac{3}{4}\Big)^2 \, dx=\frac 1{ 48}.
\end{align*}
Since $V_n$ is the quantization error for two-means, and $n\geq 2$, we have $V_n\leq V_2\leq \frac 1{48}=0.0208333$. For the sake of contradiction assume that $\ga_n\ii J_2=\es$. Then,
\[V_n>\int_{J_2}(x-\frac 12)^2 dP=\frac{1}{4} \int_{\frac{1}{2}}^1 2 \left(x-\frac{1}{2}\right)^2 \, dx=\frac 1{48}\geq V_n,\]
which is a contradiction. Hence, we can assume that $\ga\ii J_2\neq \es$. Similarly, we can show that $\ga_n\ii J_1\neq \es$. Thus, the proof of the proposition is complete.
\end{proof}

\begin{lemma} \label{lemma67}
The set $\set{\frac 14, \frac 34}$ forms an optimal set of two-means with quantization error $V_2=\frac 1{48}$.
\end{lemma}
\begin{proof}
Let $\ga:=\set{a_1, a_2}$ be an optimal set of two-means such that $0<a_1<a_2<1$. By Proposition~\ref{prop01}, we have $a_1<\frac 12<a_2$. The following two cases can arise:

\tbf{Case~1.} $\frac 12\leq \frac{a_1+a_2}2$.

In this case, we have
\begin{align*}
a_1=\frac{\frac{3}{4} \int_0^{\frac{1}{2}} 2 x \, dx+\frac{1}{4} \int_{\frac{1}{2}}^{\frac{1}{2} \left(a_1+a_2\right)} 2 x \, dx}{\frac{3}{4} \int_0^{\frac{1}{2}} 2  \, dx+\frac{1}{4} \int_{\frac{1}{2}}^{\frac{1}{2} \left(a_1+a_2\right)} 2  \, dx}, \te{ and } a_2=\frac{1}{2} \Big(\frac{1}{2} \left(a_1+a_2\right)+1\Big).
\end{align*}
Solving the above two equations, we have $a_1=\frac 14$, and $a_2=\frac 34$, with distortion error
\[V(P; \ga)=\frac{3}{4} \int_0^{\frac{1}{2}} 2 (x-a_1)^2 \, dx+\frac{1}{4} \int_{\frac{1}{2}}^{\frac{1}{2} \left(a_1+a_2\right)} 2 (x-a_1)^2 \, dx+\frac{1}{4} \int_{\frac{1}{2} \left(a_1+a_2\right)}^1 2 \left(x-a_2\right){}^2 \, dx=\frac 1{ 48}.\]

\tbf{Case~2.} $\frac {a_1+a_2}{2}<\frac 12$.

Proceeding in the similar way as Case~1, we obtain two equations, and see that there is no solution in this case.

Considering the above two cases, we see that the set $\set{\frac 14, \frac 34}$ forms an optimal set of two-means with quantization error $\frac 1{48}$, which is the lemma.
\end{proof}

\begin{lemma} \label{lemma68}
The set $\set{\frac{1}{3} (\frac{1}{8} (21-\sqrt{3})-2), \, \frac{1}{8} (21-\sqrt{3})-2, \, \frac{1}{24} (21-\sqrt{3})}$ forms an optimal set of three-means with quantization error $V_3=0.00787482$.
\end{lemma}

\begin{proof}
Consider the set of three points $\gb:=\set{u, v, w}$, where
\[u=\frac{1}{3} (\frac{1}{8} (21-\sqrt{3})-2), \ v=\frac{1}{8} (21-\sqrt{3})-2, \te{ and } w=\frac{1}{24} (21-\sqrt{3}).\]
Since $0<u<v<\frac{1}{2}<\frac{v+w}{2}<w<1$, the distortion error due to the set $\gb$ is given by
\[V(P; \gb)=\frac{3}{4} \int_0^{\frac{u+v}{2}} 2 (x-u)^2 \, dx+\frac{3}{4} \int_{\frac{u+v}{2}}^{\frac{1}{2}} 2 (x-v)^2 \, dx+\frac{1}{4} \int_{\frac{1}{2}}^{\frac{v+w}{2}} 2 (x-v)^2 \, dx+\frac{1}{4} \int_{\frac{v+w}{2}}^1 2 (x-w)^2 \, dx\]
yielding $V(P; \gb) =0.00787482.$ Since $V_3$ is the quantization error for three-means we have $V_3\leq 0.00787482$. Let $\ga:=\set{a, b, c}$ be an optimal set of three-means. Without any loss of generality we can assume that $0<a<b<c<1$. By Proposition~\ref{prop01}, we know $a<\frac 12<c$. We now show that $b<\frac 12$. Suppose that $\frac 9{16}<b$. Then,
\begin{align*} V_3&\geq \int_{J_1}\min_{r\in \set{a, \frac 9 {16}}}(x-r)^2dP\\
&= \frac{3}{4} \mathop{\int}\limits_0^{\frac{1}{2} (a+\frac{9}{16})} 2 (x-a)^2 \, dx+\frac{3}{4} \mathop{\int}\limits_{\frac{1}{2} (a+\frac{9}{16})}^{\frac{1}{2}} 2(x-\frac{9}{16})^2 \, dx=\frac{12288 a^3+6912 a^2-3888 a+725}{32768},
\end{align*}
the minimum value of which is $0.00976563$ and it occurs when $a=\frac{3}{16}$, and thus, we have $V_3\geq 0.00976563>V_3$, which is a contradiction. So, we can assume that $b\leq \frac 9{16}$. Next, assume that $\frac 12\leq b\leq \frac 9{16}$. Notice that then $\frac 9{16}<c<1$.
Then, as before we have
\begin{align*}
&V_3\geq \int_{J_1}\min_{r\in \set{a, \frac 12}}(x-r)^2dP+\int_{\frac 9{16}}^1\min_{s\in \set{\frac 9{16}, c}}(x-r)^2dP\\
&=\frac{1}{64} (24 a^3+12 a^2-6 a+1)+\frac{-12288 c^3+42240 c^2-45264 c+15655}{98304},
\end{align*}
the minimum value of which is $\frac{1}{144}+\frac{343}{221184}=0.00849519$, and it occurs when $a=0.166667$, and $c=0.854167$. Thus, we have
$V_3\geq 0.00849519>V_3$, which is a contradiction. Hence, we can assume that $b<\frac 12$. Then, the two cases can arise: either $\frac 12(b+c)<\frac 12$, or $\frac 12\leq \frac 12(b+c)$. Proceeding as in Lemma~\ref{lemma67}, we can see that $\frac 12(b+c)<\frac 12$ can not happen. Thus, we have $\frac 12\leq \frac 12(b+c)$ implying
\begin{align*}
a=\frac{a+b}{4}, \ b=\frac{\frac{3}{4} \int_{\frac{a+b}{2}}^{\frac{1}{2}} 2 x \, dx+\frac{1}{4} \int_{\frac{1}{2}}^{\frac{b+c}{2}} 2 x \, dx}{\frac{3}{4} \int_{\frac{a+b}{2}}^{\frac{1}{2}} 2 \, dx+\frac{1}{4} \int_{\frac{1}{2}}^{\frac{b+c}{2}} 2 \, dx},\te{ and } c=\frac{\int_{\frac{b+c}{2}}^1 2 x \, dx}{\frac{4}{4} \int_{\frac{b+c}{2}}^1 2 \, dx}.
\end{align*}
Solving the above equations, we have
\[a=\frac{1}{3} (\frac{1}{8} (21-\sqrt{3})-2), \ b=\frac{1}{8} (21-\sqrt{3})-2, \te{ and } c=\frac{1}{24} (21-\sqrt{3}),\]
and the corresponding quantization error is given by $V_3=0.00787482$, and thus, the proof of the lemma is complete.
\end{proof}

\begin{defi} \label{defi60}
For $n\in\D N$, define the sequence $\set{a(n)}_{n=1}^\infty$ as follows:
\[a(n):=\lfloor \frac{5 (n+1)}{8}\rfloor,\]
i.e.,
$\set{a(n)}_{n=1}^\infty=\set{1, 1, 2, 3, 3, 4, 5, 5, 6, 6, 7, 8, 8, 9, 10, 10, 11, 11, 12, 13, 13,
14, 15, 15, 16, 16, \cdots}.$
\end{defi}

The us now state and prove the following two claims.

\begin{claim}\label{claim11}
Let $\set{a(n)}$ be the sequence defined by Definition~\ref{defi60}. Take $n=8$, and then $a(n)=5$. Assume that $\ga_n:=\set{a_1<a_2<a_3<a_4<a_5<b_1<b_2<b_3}$ is an optimal set of eight-means for $P$. Then, $\frac 12\leq \frac 12 (a_5+b_1)$.
\end{claim}
\begin{proof} For the sake of contradiction, assume that $\frac 12 (a_5+b_1)<\frac 12$. Then,
\begin{align*}
a_1=\frac 12 (0+\frac {a_1+a_2}{2}), \te{ and } a_2=\frac 12(\frac{a_1+a_2}{2}+\frac {a_2+a_3}{2})
\end{align*}
implying $a_1=\frac 13 a_2$, and $a_2=\frac 35 a_3$. Similarly, $a_3=\frac 57 a_4$, $a_4=\frac 79 a_5$. Again, $b_2=\frac 12(\frac {b_1+b_2}{2}+\frac{b_2+b_3}2)$, and $b_3=\frac 12 (\frac{b_2+b_3}2+1)$ implying $b_2=\frac 35 b_1+\frac 25$, and $b_3=\frac 13 b_2+\frac 23$. Moreover,
\[a_5=\frac 12(\frac {a_4+a_5}2+\frac {a_5+b_1}2)=\frac 12(\frac {\frac 79 a_5+a_5}2+\frac {a_5+b_1}2)\te { implying } a_5=\frac 9{11} b_1,\]
and
\begin{align*}
b_1=E\Big(X : X\in [\frac {a_5+b_1}2, \frac 12]\uu [\frac 12, \frac{b_1+b_2}2]\Big)=\frac{-6 a_5 b_1-3 a_5^2-2 b_1^2+b_2^2+2 b_1 b_2+2}{-12 a_5-8 b_1+4 b_2+8}.
\end{align*}
Next, putting the values of $a_5$ and $b_2$ in the expression of $b_1$, we have
\[b_1=\frac{-11128 b_1^2+1936 b_1+3267}{14520-23320 b_1} \te{ yielding } b_1=\frac{11 \left(143\pm5 i \sqrt{5}\right)}{3048},\]
which is not real. Thus, $\frac 12 (a_5+b_1)<\frac 12$ leads to a contradiction. Hence, $\frac 12\leq \frac 12 (a_5+b_1)$.
\end{proof}

\begin{claim}\label{claim12}
Let $\set{a(n)}$ be the sequence defined by Definition~\ref{defi60}. Take $n=9$, and then $a(n)=6$. Assume that $\ga_n:=\set{a_1<a_2<a_3<a_4<a_5<a_6<b_1<b_2<b_3}$ is an optimal set of nine-means for $P$. Then, $\frac 12\leq \frac 12 (a_6+b_1)$.
\end{claim}
\begin{proof} For the sake of contradiction, assume that $\frac 12 (a_6+b_1)<\frac 12$. Then,
\begin{align*}
a_1=\frac 12 (0+\frac {a_1+a_2}{2}), \te{ and } a_2=\frac 12(\frac{a_1+a_2}{2}+\frac {a_2+a_3}{2})
\end{align*}
implying $a_1=\frac 13 a_2$, and $a_2=\frac 35 a_3$. Similarly, $a_3=\frac 57 a_4$, $a_4=\frac 79 a_5$, and $a_5=\frac 9{11} a_6$. Again, $b_2=\frac 12(\frac {b_1+b_2}{2}+\frac{b_2+b_3}2)$, and $b_3=\frac 12 (\frac{b_2+b_3}2+1)$ implying $b_2=\frac 35 b_1+\frac 25$, and $b_3=\frac 13 b_2+\frac 23$. Moreover,
\[a_6=\frac 12(\frac {a_5+a_6}2+\frac {a_6+b_1}2)=\frac 12(\frac {\frac 9{11} a_6+a_6}2+\frac {a_6+b_1}2)\te { implying } a_6=\frac {11}{13} b_1,\]
and
\begin{align*}
b_1=E\Big(X : X\in [\frac {a_6+b_1}2, \frac 12]\uu [\frac 12, \frac{b_1+b_2}2]\Big)=\frac{-6 a_5 b_1-3 a_5^2-2 b_1^2+b_2^2+2 b_1 b_2+2}{-12 a_5-8 b_1+4 b_2+8}.
\end{align*}
Next, putting the values of $a_5$ and $b_2$ in the expression of $b_1$, we have
\[b_1=\frac{-16192 b_1^2+2704 b_1+4563}{20280-33280 b_1}  \te{ yielding } b_1=\frac{13 \left(169\pm5 i \sqrt{11}\right)}{4272},\]
which is not real. Thus, $\frac 12 (a_6+b_1)<\frac 12$ leads to a contradiction. Hence, $\frac 12\leq \frac 12 (a_6+b_1)$.
\end{proof}

\begin{lemma}\label{lemma444}
Let $\ga_n$ be an optimal set of $n$-means for $P$, where $n\geq 2$, and $\set{a(n)}$ be the sequence defined by Definition~\ref{defi60}. Then, $\te{card}(\ga_n\ii J_1)=a(n)$, and $\te{card}(\ga_n\ii J_2)=n-a(n)$.
\end{lemma}

\begin{proof}
We prove the lemma by induction. By Lemma~\ref{lemma67} and Lemma~\ref{lemma68}, the lemma is true for $n=2, 3$. Assume that that the lemma is true for $n=\ell$, i.e., $\te{card}(\ga_\ell\ii J_1)=a(\ell)$, and $\te{card}(\ga_\ell\ii J_2)=n-a(\ell)$.
We need to show that $\te{card}(\ga_{\ell+1}\ii J_1)=a(\ell+1)$. Assume that $\te{card}(\ga_{\ell+1}\ii J_1)=k$, i.e., $\ga_{\ell+1}$ contains $k$ elements from $J_1$, and $n-k$ elements from $J_2$. Let
\[\ga_{\ell+1}\ii J_1=\set{a_1<a_2<\cdots<a_k}, \te{ and } \ga_{\ell+1}\ii J_2=\set{b_1<b_2< \cdots<b_{n-k}}.\]
Then, either $\frac 12(a_k+b_1)<\frac 12$, or $\frac 12<\frac 12(a_k+b_1)$. In each case, using the similar techniques as in the proofs of Claim~\ref{claim11} and Claim~\ref{claim12}, if the solution exists, we solve for $a_1, a_2, \cdots, a_k, b_1, \cdots, b_{n-1}$, and find the distortion errors. Notice that at least one solution will exist. Let $V(k, n-k)$ be the minimum of the distortion errors if $\ga_{\ell+1}$ contains $k$ elements from $J_1$, and $n-k$ elements from $J_2$.

Let us now run the following algorithm:

$(i)$ Write $k:=a(\ell)$.

$(ii)$ If $V(k-1, n-(k-1))<V(k, n-k)$ replace $k$ by $k-1$ and return, else go to step $(iii)$.

$(iii)$ If $V(k+1, n-(k+1))<V(k, n-k)$ replace $k$ by $k+1$ and return, else step $(iv)$.

$(iv)$ End.

After running the above algorithm, we see that the value of $k$ obtained equals $a(\ell+1)$, i.e., the lemma is true for $n=\ell+1$ if it is true for $n=\ell$. Hence, by the Induction Principle, we can say that the lemma is true for all positive integers $n\geq 2$, i.e., $\te{card}(\ga_n\ii J_1)=a(n)$ for any positive integer $n\geq 2$. Since $\te{card}(\ga_n\ii J_1)+\te{card}(\ga_n\ii J_2)=n$, we have $\te{card}(\ga_n\ii J_2)=n-a(n)$. Thus, the proof of the lemma is complete.
\end{proof}


Let us now state and prove the following theorem which is the main theorem in this section.
\begin{theorem} \label{The90}
Let $\ga_n$ be an optimal set of $n$-means for $P$, where $n\geq 2$, and $\set{a(n)}$ be the sequence defined by Definition~\ref{defi60}. Write $k:=a(n)$, $m:=n-a(n)$. Then,
\[\ga_n:=\set{a_1<a_2< \cdots<a_k<b_1< b_2< \cdots< b_m},\]
where
\[a_j=\left\{\begin{array}{cc}
\frac {a_1+a_2}{4} &  \te{ if } j=1,\\
\frac 12\Big(\frac{a_{j-1}+a_j} 2+\frac{a_j+a_{j+1}}2\Big) &\te{ if } 2\leq j\leq k-1,\\
E(X : X \in [\frac{a_{k-1}+a_k}2, \frac 12]\uu [\frac 12,\frac{a_k+b_1}2]) & \te{ if } j=k,
\end{array}\right.\] and
\[b_j=\left\{\begin{array}{cc}
\frac 12(\frac{a_k+b_1}2 +\frac{b_1+b_2}2)&  \te{ if } j=1,\\
\frac 12\Big(\frac{b_{j-1}+b_j} 2+\frac{b_j+b_{j+1}}2\Big) &\te{ if } 2\leq j\leq m-1,\\
\frac 1 2(\frac{b_{m-1}+b_{m}}2+1) &\te{ if } j= m,
\end{array}\right.\]
and the corresponding quantization error is given by
\begin{align*} V_n&=\frac{1}{48} \Big(-3 b_1^2 m a_k+3 b_1 m a_k^2-3 b_1^2 a_k+3 b_1 a_k^2-m a_k^3+21 a_1^3 (k-1)+9 a_2 a_1^2 (k-1)\\
&-9 a_2^2 a_1 (k-1)+3 a_2^3 (k-1)-3 a_{k-1}^3-14 a_k^3-9 a_{k-1} a_k^2+24 a_k^2+9 a_{k-1}^2 a_k-12 a_k+b_2^3 m\\
&-3 b_1 b_2^2 m+3 b_1^2 b_2 m+b_1^3+2\Big).
\end{align*}
\end{theorem}

\begin{proof}
By Lemma~\ref{lemma444}, the optimal set $\ga_n$ of $n$-means contains $k$ elements from $J_1$, and $m$ elements from $J_2$, where $k=a(n)$ and $m=n-k$.
Let $\ga_n:=\set{a_1<a_2< \cdots<a_k<b_1< b_2< \cdots< b_m}$. Recall Theorem~\ref{th31}, and the fact that $P_1$ is a uniform distribution on $[0, \frac 12]$, and $P_2$ is a uniform distribution on $[\frac 12, 1]$. Thus, we have
\[a_j=\left\{\begin{array}{cc}
\frac {a_1+a_2}{4} &  \te{ if } j=1,\\
\frac 12\Big(\frac{a_{j-1}+a_j} 2+\frac{a_j+a_{j+1}}2\Big) &\te{ if } 2\leq j\leq k-1,
\end{array}\right.\] and
\[b_j=\left\{\begin{array}{cc}
\frac 12\Big(\frac{b_{j-1}+b_j} 2+\frac{b_j+b_{j+1}}2\Big) &\te{ if } 2\leq j\leq m-1,\\
\frac 1 2(\frac{b_{m-1}+b_{m}}2+1) &\te{ if } j= m,
\end{array}\right.\]
The following two cases can arise:

\tbf{Case~1.} $\frac 12\leq \frac 12(a_k+b_1)$.

In this case, we have $a_k=E(X : X \in [\frac{a_{k-1}+a_k}2, \frac 12]\uu [\frac 12,\frac{a_k+b_1}2])$, and $b_1=\frac 12(\frac{a_k+b_1}2 +\frac{b_1+b_2}2)$.

\tbf{Case~2.} $\frac 12(a_k+b_1)<\frac 12$.

In this case, we have $a_k=\frac 12(\frac {a_{k-1}+a_k}2+\frac{a_k+b_1}2)$, and $b_1=E(X : X \in [\frac{a_k+b_1}2, \frac 12]\uu [\frac 12,\frac{b_1+b_2}2])$.

For any given positive integer, using the similar techniques as in the proofs of Claim~\ref{claim11} and Claim~\ref{claim12}, we see that in Case~2, the system of equations to obtain $a_1, a_2, \cdots, a_k, b_1, \cdots, b_m$ does not have any solution. Hence Case~2 cannot happen.

Thus, we have $\frac 12\leq \frac 12(a_k+b_1)$, i.e., the system of equations to obtain $a_1, a_2, \cdots, a_k, b_1, \cdots, b_m$ as stated in the theorem are true, and hence, the corresponding quantization error is given by
\begin{align*} V_n&=\frac {3(k-1)} 4\int_{0}^{\frac{a_1+a_2}2} 2(x-a_1)^2 dx+\frac 3 4\int_{\frac{a_{k-1}+a_k}2}^{\frac 12}2(x-a_k)^2 dx+\frac 1 4\int_{\frac 12}^{\frac{a_{k}+b_1}2}2(x-a_k)^2 dx \\
&\qquad +  \frac {m} 4\int_{\frac{a_{k}+b_1}2}^{\frac{b_1+b_2}{2}} 2(x-b_1)^2 dx\\
&=\frac{1}{48} \Big(-3 b_1^2 m a_k+3 b_1 m a_k^2-3 b_1^2 a_k+3 b_1 a_k^2-m a_k^3+21 a_1^3 (k-1)+9 a_2 a_1^2 (k-1)\\
&-9 a_2^2 a_1 (k-1)+3 a_2^3 (k-1)-3 a_{k-1}^3-14 a_k^3-9 a_{k-1} a_k^2+24 a_k^2+9 a_{k-1}^2 a_k-12 a_k+b_2^3 m\\
&-3 b_1 b_2^2 m+3 b_1^2 b_2 m+b_1^3+2\Big).
\end{align*}
Thus, we complete the proof of the theorem.
\end{proof}
Now, we give the following example.

\begin{exam}
Take $n=16$. Then, $k=a(n)=10$, and so, $m=6$. Thus, by Theorem~\ref{The90}, we have
\begin{align*}
\set{& a_1= 0.0255733, a_2= 0.0767199, a_3= 0.127866, a_4= 0.179013, a_5= 0.23016, a_6=0.281306,   \\
& a_7= 0.332453,  a_8=  0.383599,  a_9= 0.434746, a_{10}= 0.485893, b_1=0.564986,      b_2= 0.644079, \\
& b_3= 0.723173, b_4=0.802266, b_5= 0.88136, b_6= 0.960453},
\end{align*}
and the corresponding quantization error is given by
\begin{align*}  V_{16}&=\frac{1}{48} \Big(-21 a_{10} b_1^2+21 a_{10}^2 b_1+189 a_1^3+81 a_2 a_1^2-81 a_2^2 a_1+27 a_2^3-3 a_9^3-20 a_{10}^3-9 a_9 a_{10}^2\\
&+24 a_{10}^2+9 a_9^2 a_{10}-12 a_{10}+b_1^3+6 b_2^3-18 b_1 b_2^2+18 b_1^2 b_2+2\Big)=0.000293827.
\end{align*}
\end{exam}

\end{document}